\documentclass[twoside,12pt,letter]{article}

\usepackage[margin=0.7in]{geometry}
\usepackage{amsmath, amssymb, amsthm}
\usepackage{bm}
\usepackage{hyperref}
\usepackage{enumitem}
\usepackage[dvipsnames]{xcolor}
\usepackage{setspace}
\usepackage{float}
\usepackage{graphicx}
\usepackage{tikz}
\usepackage{cite}
\usetikzlibrary{arrows.meta}

\hypersetup{
colorlinks=true,
linkcolor=MidnightBlue,
urlcolor=MidnightBlue,
citecolor=MidnightBlue
}

\theoremstyle{plain}
\newtheorem{theorem}{Theorem}[section]
\newtheorem{proposition}[theorem]{Proposition}
\newtheorem{lemma}[theorem]{Lemma}
\newtheorem{corollary}[theorem]{Corollary}

\theoremstyle{definition}
\newtheorem{example}[theorem]{Example}

\theoremstyle{remark}
\newtheorem{remark}[theorem]{Remark}

\numberwithin{equation}{section}

\title{A capacitary approach to Lyapunov-type inequalities for elliptic problems on weighted graphs}

\author{Mohamed Jleli\textsuperscript{1}, Bessem Samet\textsuperscript{1,*}}

\date{}

\begin{document}

\maketitle

\begingroup
\renewcommand{\thefootnote}{1}
\footnotetext{Department of Mathematics, College of Science, King Saud University, Riyadh 11451, Saudi Arabia.}

\renewcommand{\thefootnote}{*}
\footnotetext{Corresponding author.\\ E-mail addresses: jleli@ksu.edu.sa (M. Jleli), bsamet@ksu.edu.sa (B. Samet)}
\endgroup

\begin{abstract}
We initiate the study of Lyapunov-type inequalities for Dirichlet problems
driven by the discrete \(p\)-Laplacian on weighted graphs. The approach is
capacitary and is based on point \(p\)-capacities and the associated
capacitary radii. First, we prove general Lyapunov-type inequalities on
arbitrary connected locally finite weighted graphs. These inequalities provide
intrinsic lower bounds, expressed in terms of the capacitary radii, for the
positive part of the potential whenever the corresponding Dirichlet problem
admits a nontrivial solution. Next, we estimate these capacitary radii in
several geometric settings and prove the sharpness of the resulting
Lyapunov-type inequalities. As an application, we derive lower bounds for the
first weighted Dirichlet eigenvalue of the discrete \(p\)-Laplacian.
\end{abstract}
\begin{small} 
Dirichlet eigenvalues

\medskip 

\end{small}

\section{Introduction}

Lyapunov-type inequalities play an important role in the qualitative theory of
ordinary differential equations. They provide necessary conditions for the
existence of nontrivial solutions of boundary value problems and are closely
related to stability, oscillation, disconjugacy, and eigenvalue estimates.

In its classical form, Lyapunov's inequality can be stated as follows. If
\(w\in C([a,b])\) and the boundary value problem
\[
\left\{
\begin{array}{ll}
-y''(t)=w(t)y(t), & t\in(a,b),\\[2mm]
y(a)=y(b)=0,
\end{array}
\right.
\]
admits a nontrivial solution, then
\begin{equation}\label{ST-LYINQ}
\frac{4}{b-a}<\int_a^b |w(t)|\,dt.
\end{equation}
The constant \(4\) in \eqref{ST-LYINQ} is sharp, in the sense that it cannot
be replaced by any larger constant. Moreover, by the Sturm comparison theorem, the
term \(|w|\) can be replaced by the positive part \(w_+\), giving
\[
\frac{4}{b-a}<\int_a^b w_+(t)\,dt.
\]

Although \eqref{ST-LYINQ} is commonly attributed to Lyapunov, the original
result of Lyapunov~\cite{Lyapunov} was formulated in terms of the stability of periodic
differential equations.  Borg~\cite{Borg} later gave what appears to be the first proof of Lyapunov's
inequality in the literature, in an equivalent form due to Beurling. He then
used this inequality to provide a shorter proof of Lyapunov's stability theorem. For further developments of Lyapunov-type inequalities for ordinary
differential equations, together with several applications, we refer to the
survey by Brown and Hinton~\cite{BrownHinton2000}.

Lyapunov's inequality has been considerably generalized and has inspired numerous
extensions to different classes of problems, including nonlinear ordinary
differential equations. For instance, Pinasco~\cite{Pinasco-04} studied the
one-dimensional \(p\)-Laplacian problem
\[
\left\{
\begin{array}{ll}
-\bigl(|u'|^{p-2}u'\bigr)'=w(t)|u|^{p-2}u, & t\in(a,b),\\[2mm]
u(a)=u(b)=0,
\end{array}
\right.
\]
where \(1<p<\infty\) and \(w\) is a bounded positive function. He proved that
if this problem admits a nontrivial weak solution \(u\in W_0^{1,p}(a,b)\), then
\[
\frac{2^p}{(b-a)^{p-1}}
\le
\int_a^b w(t)\,dt.
\]
For \(p=2\), the above inequality  coincides with \eqref{ST-LYINQ} in the case \(w\ge0\).

Other one-dimensional nonlinear extensions involve nonlinear functional
differential equations~\cite{Eliason}, nonlinear operators in Orlicz
spaces~\cite{deNapoliPinasco2005}, the \(\psi\)-Laplacian
operator~\cite{Sanchez}, and the Minkowski-curvature
operator~\cite{Liu,Liu2,Yang-19}.

Lyapunov-type inequalities have also been studied in higher-dimensional
settings. These include linear partial differential equations with Neumann
boundary conditions~\cite{Canada}, \(p\)-Laplacian equations with Dirichlet
boundary conditions~\cite{deNapoliPinasco2016}, quasilinear elliptic
systems~\cite{deNapoliPinasco2006}, elliptic problems with Robin boundary
conditions~\cite{HA}, extremal Pucci equations~\cite{Tyagi}, and fractional
partial differential equations~\cite{JKS}. A survey of several results on
Lyapunov-type inequalities in both one-dimensional and higher-dimensional
settings can be found in the monograph~\cite{Pinasco-13}.

Results closely related to the present work were obtained by de N\'apoli and
Pinasco~\cite{deNapoliPinasco2016}.  They studied Lyapunov-type inequalities for the Dirichlet problem
\begin{equation}\label{LY-deP}
\left\{
\begin{array}{ll}
-\Delta_p u=w(x)|u|^{p-2}u, & x\in\Omega,\\[2mm]
u=0, & x\in\partial\Omega,
\end{array}
\right.
\end{equation}
where \(1<p<\infty\), \(\Omega\) is an open subset of \(\mathbb R^N\), and
\(w\) is a nonnegative measurable potential. Here
\[
\Delta_p u=\operatorname{div}\bigl(|\nabla u|^{p-2}\nabla u\bigr).
\]
In the case \(p>N\), they proved a lower bound for the \(L^1\)-norm of the
potential \(w\) in terms of the inner radius \(r_\Omega\) of the domain
\(\Omega\). More precisely, if the Dirichlet problem \eqref{LY-deP} admits a
nontrivial solution \(u\in W_0^{1,p}(\Omega)\), then
\[
\frac{C}{r_\Omega^{p-N}}
\le
\|w\|_{L^1(\Omega)},
\]
where \(C=C(p,N)>0\) is independent of \(\Omega\).  In the subcritical case \(1<p<N\), they obtained an \(L^s\)-version of the
Lyapunov inequality. More precisely, if \(s>N/p\) and the Dirichlet problem
\eqref{LY-deP} admits a nontrivial solution \(u\in W_0^{1,p}(\Omega)\), then
\begin{equation}\label{Pin-plN}
\frac{C}{r_\Omega^{p-\frac Ns}}
\le
\|w\|_{L^s(\Omega)},
\end{equation}
where \(C>0\) depends on \(p\), \(N\), and the capacity of
\(\mathbb R^N\setminus\Omega\).

Recently, elliptic problems on weighted graphs have attracted increasing
attention, motivated by their connections with discrete analysis, network
models, and the geometry of graphs. In particular, questions concerning the
existence, nonexistence, and qualitative properties of solutions to elliptic
equations and inequalities on graphs have been studied in several settings;
see, for instance,
\cite{DuongMinh2025,Grigoryan2018,Grigoryan2016,Gu,Meglioli,MonticelliPunzoSomaglia2025}.

To the best of our knowledge, Lyapunov-type inequalities on weighted graphs
have not yet been investigated. The aim of the present paper is to initiate
such a study for Dirichlet problems driven by the discrete \(p\)-Laplacian,
from a capacitary point of view. More precisely, we consider the problem
\[
\left\{
\begin{array}{ll}
-\Delta_p^\omega u=w(x)|u|^{p-2}u, & x\in\Omega,\\[2mm]
u=0, & x\in V\setminus\Omega,
\end{array}
\right.
\]
where \(1<p<\infty\), \((V,\omega,\mu)\) is a connected locally finite
weighted graph, \(\Omega\) is a nonempty finite subset of \(V\),
\(w:\Omega\to\mathbb R\) is a potential, and \(\Delta_p^\omega\) denotes the
discrete \(p\)-Laplacian, defined by
\[
\Delta_p^\omega u(x)
=
\frac{1}{\mu(x)}
\sum_{y\sim x}
\omega_{xy}|u(y)-u(x)|^{p-2}\bigl(u(y)-u(x)\bigr).
\]

The main feature of our approach is the introduction of capacitary radii
adapted to finite subsets of weighted graphs. These radii are defined in terms
of point \(p\)-capacities and allow us to establish general Lyapunov-type
inequalities on arbitrary connected locally finite weighted graphs.
Next, we estimate these capacitary quantities in several geometric settings and
prove the sharpness of the resulting Lyapunov-type inequalities. 

First, under
bounded-geometry, polynomial volume-growth, and relative isoperimetric
assumptions, we obtain an inner-radius Lyapunov-type inequality in the
supercritical range \(p>D\), where \(D\) is the volume-growth exponent. We also
show that the power of the inner radius in this estimate is sharp. 

Next, we study radial path graphs with polynomial edge weights and polynomial
vertex measures, where \(D\) plays the role of a volume-growth dimension, and
derive explicit Lyapunov-type inequalities in the three regimes \(p>D\),
\(p=D\), and \(1<p<D\). In this setting, we prove the optimality of the
corresponding power-type and logarithmic scales. 

Finally, we consider \(q\)-regular trees, \(q\ge3\), and obtain a uniform Lyapunov-type inequality with an explicit sharp constant.
 
As an application of the capacitary Lyapunov-type inequalities, we derive lower
bounds for the first weighted Dirichlet eigenvalue of the discrete
\(p\)-Laplacian on finite subsets of weighted graphs. These estimates are
expressed in terms of capacitary radii and are then specialized to the
geometric settings treated in this paper.

The rest of the paper is organized as follows. In Section~\ref{sec2}, we fix
notation, recall the basic definitions for weighted graphs and difference
operators, and introduce the capacitary radii used throughout the paper. In
Section~\ref{sec3}, we state the main Lyapunov-type inequalities and their
geometric consequences. In Section~\ref{sec4}, we apply these inequalities to
weighted Dirichlet eigenvalue problems. In Section~\ref{sec5}, we prove the
main results.

\section{Mathematical framework}\label{sec2}

In this section, we fix the notation and recall the basic tools that will be
used throughout the paper. We also introduce the capacitary quantities needed
in our estimates.

\subsection{The graph setting}

First, we recall some basic notions concerning weighted graphs. For further
details, we refer the reader to \cite{Grigoryan2018}. Throughout this paper,
\(\mathbb{N}\) denotes the set of positive integers, and we set
\(\mathbb{N}_0=\mathbb{N}\cup\{0\}\).

A weighted graph is a triple \((V,\omega,\mu)\), where \(V\) is a countably
infinite set, \(\mu:V\to(0,\infty)\) is a vertex measure, and
\(\omega:V\times V\to[0,\infty)\) is a symmetric edge-weight function satisfying
\[
\omega_{xy}=\omega_{yx},\qquad \omega_{xx}=0,\qquad x,y\in V.
\]

For a weighted graph \((V,\omega,\mu)\), we use the following terminology.
\begin{itemize}
\item Two vertices \(x,y\in V\) are said to be adjacent if \(\omega_{xy}>0\).
In this case, we write \(x\sim y\), and the unordered pair \(\{x,y\}\) is called
an edge of the graph.

\item The degree of a vertex \(x\in V\) is defined by
\[
\deg(x)=\#\{y\in V:\ y\sim x\}.
\]

\item A path from \(x\) to \(y\) is either the trivial path \(x=y\), of length
\(0\), or a finite sequence of vertices
\[
x=x_0,x_1,\ldots,x_m=y,
\]
where \(m\in\mathbb{N}\), such that \(x_{i-1}\sim x_i\) for every
\(i=1,\ldots,m\). The integer \(m\) is called the length of the path.

\item A path
\[
x=x_0,x_1,\ldots,x_m=y
\]
is said to be simple if the vertices \(x_0,x_1,\ldots,x_m\) are pairwise
distinct.

\item A cycle is a finite sequence of vertices
\[
x_0,x_1,\ldots,x_m=x_0,
\]
where \(m\ge 3\), such that \(x_0,x_1,\ldots,x_{m-1}\) are pairwise distinct and
\(x_{i-1}\sim x_i\) for every \(i=1,\ldots,m\).
\end{itemize}

The symmetry condition \(\omega_{xy}=\omega_{yx}\) implies that adjacency is
symmetric; hence the graph is undirected. Moreover, since \(\omega_{xx}=0\) for
every \(x\in V\), loops are excluded.

A weighted graph \((V,\omega,\mu)\) is said to be locally finite if
\[
\deg(x)<\infty,\qquad x\in V.
\]
It is said to be connected if, for every \(x,y\in V\), there exists a path from
\(x\) to \(y\).

Let \((V,\omega,\mu)\) be a connected locally finite weighted graph. The graph
distance between two vertices \(x,y\in V\) is defined by
\[
d(x,y)=\min\left\{m\in\mathbb{N}_0:\ \text{there exists a path of length }m
\text{ from }x\text{ to }y\right\}.
\]
This minimum is well defined because the graph is connected. For \(x\in V\)
and \(R\ge 0\), the ball centered at \(x\) with radius \(R\) is defined by
\[
B(x,R)=\{y\in V:\ d(x,y)\le R\}.
\]
Since the graph is locally finite, every ball \(B(x,R)\)  is finite.

\medskip

\noindent\textbf{Notation.}
Let \((V,\omega,\mu)\) be a connected locally finite weighted graph.
\begin{itemize}
\item For every finite set \(A\subset V\), we set
\[
\mu(A)=\sum_{x\in A}\mu(x),
\]
with the convention that \(\mu(\emptyset)=0\).

\item Let \(\Omega\subset V\) be a nonempty subset. For \(x\in V\), we define
\[
d(x,V\setminus\Omega)=\inf_{y\in V\setminus\Omega}d(x,y).
\]
In particular, if \(x\in\Omega\), then
\[
d(x,V\setminus\Omega)\ge1.
\]
The inner radius of \(\Omega\) is defined by
\begin{equation}\label{r-Omega}
r_\Omega=\sup_{x\in\Omega}d(x,V\setminus\Omega).
\end{equation}
If \(\Omega\) is finite, then
\[
1\le r_\Omega<\infty.
\]

\item For every nonempty finite set \(A\subset V\) and every function
\(u:V\to\mathbb R\), we denote by \(u_A\) the weighted average of \(u\) over
\(A\), namely
\[
u_A
=
\frac{1}{\mu(A)}
\sum_{x\in A}\mu(x)u(x).
\]

\item For \(s\in\mathbb R\), we denote by \(s_+\) its positive part, namely
\[
s_+=\max\{s,0\}.
\]
More generally, if \(A\subset V\) and \(f:A\to\mathbb R\), we define
\(f_+:A\to[0,\infty)\) by
\[
f_+(x)=\bigl(f(x)\bigr)_+=\max\{f(x),0\},
\qquad x\in A.
\]

\item If \(1<s<\infty\), we denote by
\[
s'=\frac{s}{s-1}
\]
the conjugate exponent of \(s\).
\end{itemize}

\subsection{Function spaces and energy}

Throughout this subsection, let \((V,\omega,\mu)\) be a connected locally finite
weighted graph.

Let \(A\subset V\) be nonempty. For \(1\le q<\infty\), we define
\[
\ell^q(A,\mu)
=
\left\{
u:A\to\mathbb{R}:\ 
\sum_{x\in A}\mu(x)|u(x)|^q<\infty
\right\},
\]
endowed with the norm
\[
\|u\|_{\ell^q(A,\mu)}
=
\left(\sum_{x\in A}\mu(x)|u(x)|^q\right)^{1/q}.
\]
We also define
\[
\ell^\infty(A)
=
\left\{
u:A\to\mathbb{R}:\ \sup_{x\in A}|u(x)|<\infty
\right\},
\]
with norm
\[
\|u\|_{\ell^\infty(A)}
=
\sup_{x\in A}|u(x)|.
\]
When \(A\) is finite, this norm is simply
\[
\|u\|_{\ell^\infty(A)}
=
\max_{x\in A}|u(x)|.
\]

Let \(1<p<\infty\). For a function \(u:V\to\mathbb{R}\), we define its
\(p\)-energy by
\[
E_p(u)
=
\frac12\sum_{x,y\in V}\omega_{xy}|u(y)-u(x)|^p,
\]
whenever the sum is finite. Equivalently, since \(\omega\) is symmetric,
\[
E_p(u)
=
\sum_{\{x,y\}:x\sim y}\omega_{xy}|u(y)-u(x)|^p.
\]

Let \(\Omega\subset V\) be a nonempty finite subset. We set
\[
X_V(\Omega)
=
\{u:V\to\mathbb{R}:\ u=0\ \text{in }V\setminus\Omega\}.
\]
If \(u\in X_V(\Omega)\), then \(u\) vanishes in \(V\setminus\Omega\). Hence
only edges having at least one endpoint in \(\Omega\) can contribute to the
energy. Therefore,
\[
E_p(u)
=
\sum_{\substack{\{x,y\}:x\sim y\\ \{x,y\}\cap\Omega\neq\emptyset}}
\omega_{xy}|u(y)-u(x)|^p.
\]
Moreover, since \(\Omega\) is finite and the graph is locally finite, the set
\[
\bigl\{\{x,y\}:x\sim y,\ \{x,y\}\cap\Omega\neq\emptyset\bigr\}
\]
is finite. Consequently, every \(u\in X_V(\Omega)\) has finite \(p\)-energy.

We equip \(X_V(\Omega)\) with the energy norm
\[
\|u\|_{X_V(\Omega),p}=E_p(u)^{1/p}.
\]
This is a norm on \(X_V(\Omega)\). Indeed, homogeneity is immediate, and the
triangle inequality follows from Minkowski's inequality applied to the edge
differences. Finally, if
\[
\|u\|_{X_V(\Omega),p}=0,
\]
then \(u(y)=u(x)\) whenever \(x\sim y\). Since the graph is connected, \(u\) is
constant in \(V\). Since \(u=0\) in \(V\setminus\Omega\), this constant is zero.
Thus, \(u\equiv 0\) in \(V\).

\subsection{Difference operators and the \(p\)-Laplacian}

Let \((V,\omega,\mu)\) be a connected locally finite weighted graph, and let
\(u:V\to\mathbb{R}\). For \(x,y\in V\), we set
\[
\nabla_{xy}u=u(y)-u(x).
\]

Let \(1<p<\infty\). The weighted \(p\)-Laplacian of \(u\) is defined by
\[
\Delta_p^\omega u(x)
=
\frac{1}{\mu(x)}
\sum_{y\sim x}
\omega_{xy}|\nabla_{xy}u|^{p-2}\nabla_{xy}u,
\qquad x\in V.
\]
Since the graph is locally finite, the above sum is finite for every \(x\in V\).

In the linear case \(p=2\), the weighted \(p\)-Laplacian reduces to the weighted
Laplacian
\[
\Delta^\omega u(x)
=
\frac{1}{\mu(x)}
\sum_{y\sim x}
\omega_{xy}\bigl(u(y)-u(x)\bigr),
\qquad x\in V.
\]

We shall use the following discrete integration-by-parts identity. Let
\(\Omega\subset V\) be a nonempty finite subset. If
\(u,\varphi\in X_V(\Omega)\), then
\[
-\sum_{x\in\Omega}\mu(x)\Delta_p^\omega u(x)\varphi(x)
=
\frac12
\sum_{x,y\in V}
\omega_{xy}
|\nabla_{xy}u|^{p-2}\nabla_{xy}u\,\nabla_{xy}\varphi.
\]
In particular, taking \(\varphi=u\), we obtain
\[
-\sum_{x\in\Omega}\mu(x)\Delta_p^\omega u(x)u(x)
=
E_p(u).
\]

\subsection{A localized variational eigenvalue problem}

We shall use the following elementary variational fact several times in the
sequel.

\begin{lemma}\label{lem-variational-localized-eigenvalue}
Consider a connected locally finite weighted graph \((V,\omega,\mu)\), a
nonempty finite subset \(\Omega\subset V\), and a nonempty subset
\(A\subset\Omega\). For \(1<p<\infty\), define
\[
\lambda_A
=
\inf_{\substack{u\in X_V(\Omega)\\
\|u\|_{\ell^p(A,\mu)}>0}}
\frac{E_p(u)}
{\|u\|_{\ell^p(A,\mu)}^p}.
\]
Then the infimum is attained by some nontrivial \(u_A\in X_V(\Omega)\).
Moreover, \(u_A\) satisfies
\[
-\Delta_p^\omega u_A(x)
=
\lambda_A\mathbf{1}_A(x)|u_A(x)|^{p-2}u_A(x),
\qquad x\in\Omega.
\]
\end{lemma}

\begin{proof}
By homogeneity,
\begin{equation}\label{inf2-lambdaA}
\lambda_A
=
\inf
\left\{
E_p(u):\
u\in X_V(\Omega),\
\|u\|_{\ell^p(A,\mu)}=1
\right\}.
\end{equation}
Since \(A\neq\emptyset\), the admissible set is nonempty. Indeed, choose
\(x_0\in A\). Then the function
\[
u=\mu(x_0)^{-1/p}\mathbf{1}_{\{x_0\}}
\]
belongs to \(X_V(\Omega)\), because \(\{x_0\}\subset A\subset\Omega\).
Moreover,
\[
\|u\|_{\ell^p(A,\mu)}^p
=
\sum_{x\in A}\mu(x)|u(x)|^p
=
\mu(x_0)\mu(x_0)^{-1}
=
1.
\]
Thus, \(u\) belongs to the admissible set.

Let \((u_n)\subset X_V(\Omega)\) be a minimizing sequence such that
\begin{equation}\label{bd-un}
\|u_n\|_{\ell^p(A,\mu)}=1
\end{equation}
and
\begin{equation}\label{cv-energy}
E_p(u_n)\to \lambda_A
\qquad \text{as } n\to\infty.
\end{equation}
Since \(\Omega\) is finite, \(X_V(\Omega)\) is finite-dimensional. Moreover,
\(E_p(\cdot)^{1/p}\) is a norm on \(X_V(\Omega)\). Since \((E_p(u_n))\) is
bounded and all norms on \(X_V(\Omega)\) are equivalent, the sequence
\((u_n)\) is bounded in \(\ell^\infty(\Omega)\). Therefore, passing to a
subsequence if necessary, we may assume that
\begin{equation}\label{un-cv}
\|u_n-u_A\|_{\ell^\infty(\Omega)}\to 0
\qquad \text{as } n\to\infty,
\end{equation}
for some \(u_A\in X_V(\Omega)\).

Since \(A\) is finite, it follows from \eqref{bd-un} and \eqref{un-cv} that
\[
\begin{aligned}
1
&=
\lim_{n\to\infty}\sum_{x\in A}\mu(x)|u_n(x)|^p\\
&=
\sum_{x\in A}\mu(x)|u_A(x)|^p,
\end{aligned}
\]
that is,
\begin{equation}\label{norm-uA}
\|u_A\|_{\ell^p(A,\mu)}=1.
\end{equation}

Moreover, since \(\Omega\) is finite and the graph is locally finite, the set
\[
\bigl\{\{x,y\}:x\sim y,\ \{x,y\}\cap\Omega\neq\emptyset\bigr\}
\]
is finite. Hence, using the fact that the functions \(u_n\) and \(u_A\) vanish
in \(V\setminus\Omega\), it follows from \eqref{cv-energy} and \eqref{un-cv}
that
\[
\begin{aligned}
\lambda_A
&=
\lim_{n\to\infty}
\sum_{\substack{\{x,y\}:x\sim y\\ \{x,y\}\cap\Omega\neq\emptyset}}
\omega_{xy}|u_n(y)-u_n(x)|^p \\
&=
\sum_{\substack{\{x,y\}:x\sim y\\ \{x,y\}\cap\Omega\neq\emptyset}}
\omega_{xy}|u_A(y)-u_A(x)|^p,
\end{aligned}
\]
that is,
\begin{equation}\label{id-lambda-E}
\lambda_A=E_p(u_A).
\end{equation}
Thus, by \eqref{norm-uA} and \eqref{id-lambda-E}, the infimum in
\eqref{inf2-lambdaA} is attained by \(u_A\). In particular,
\(u_A\not\equiv 0\).

Next, we derive the Euler--Lagrange equation. Set
\[
G(u)=\|u\|_{\ell^p(A,\mu)}^p
=
\sum_{x\in A}\mu(x)|u(x)|^p.
\]
By \eqref{norm-uA}, we have
\[
G'(u_A)[u_A]
=
p\sum_{x\in A}\mu(x)|u_A(x)|^p
=
p\neq 0.
\]
Therefore, the Lagrange multiplier rule gives the existence of
\(\Lambda_A\in\mathbb{R}\) such that, for every \(\varphi\in X_V(\Omega)\),
\[
\frac12
\sum_{x,y\in V}
\omega_{xy}
|u_A(y)-u_A(x)|^{p-2}
\bigl(u_A(y)-u_A(x)\bigr)
\bigl(\varphi(y)-\varphi(x)\bigr)
=
\Lambda_A
\sum_{x\in A}
\mu(x)|u_A(x)|^{p-2}u_A(x)\varphi(x).
\]
Taking \(\varphi=u_A\) and using \eqref{norm-uA} and \eqref{id-lambda-E}, we
obtain
\[
\Lambda_A=\lambda_A.
\]
Hence, by the discrete integration-by-parts formula,
\[
-\sum_{x\in\Omega}\mu(x)\Delta_p^\omega u_A(x)\varphi(x)
=
\lambda_A
\sum_{x\in A}\mu(x)|u_A(x)|^{p-2}u_A(x)\varphi(x),
\qquad
\varphi\in X_V(\Omega).
\]
Taking \(\varphi=\mathbf{1}_{\{\xi\}}\), with \(\xi\in\Omega\), yields
\[
-\Delta_p^\omega u_A(\xi)
=
\lambda_A\mathbf{1}_A(\xi)|u_A(\xi)|^{p-2}u_A(\xi).
\]
Since \(\xi\in\Omega\) is arbitrary, the desired equation follows.
\end{proof}

\subsection{Point capacity and capacitary radii}

Let \((V,\omega,\mu)\) be a connected locally finite weighted graph. Let
\(\Omega\subset V\) be a nonempty finite subset, and let \(1<p<\infty\).

For \(x\in\Omega\), we define the point \(p\)-capacity of \(x\)
relative to \(\Omega\) by
\[
\operatorname{Cap}_p(x,\Omega)
=
\inf\left\{
E_p(\varphi):\ \varphi\in X_V(\Omega),\ \varphi(x)=1
\right\}.
\]
The admissible class is nonempty, since the function
\[
\varphi=\mathbf{1}_{\{x\}}
\]
belongs to \(X_V(\Omega)\). Hence
\[
\operatorname{Cap}_p(x,\Omega)<\infty.
\]
To prove positivity, we use the fact that \(\Omega\) is finite, so
\(X_V(\Omega)\) is finite-dimensional. Moreover, \(E_p(\cdot)^{1/p}\) is a norm
on \(X_V(\Omega)\). Hence all norms on \(X_V(\Omega)\) are equivalent. In
particular, there exists a constant \(C>0\) such that
\[
\|\varphi\|_{\ell^\infty(\Omega)}
\le
C E_p(\varphi)^{1/p},
\qquad \varphi\in X_V(\Omega).
\]
Therefore, if \(\varphi\in X_V(\Omega)\) and \(\varphi(x)=1\), then
\[
1\le C E_p(\varphi)^{1/p},
\]
and hence
\[
E_p(\varphi)\ge C^{-p}.
\]
Taking the infimum over all such \(\varphi\), we obtain
\[
\operatorname{Cap}_p(x,\Omega)\ge C^{-p}>0.
\]
Consequently,
\[
0<\operatorname{Cap}_p(x,\Omega)<\infty.
\]

We define the capacitary radius of the finite domain \(\Omega\) by
\[
\mathcal{R}_p(\Omega)
=
\max_{x\in\Omega}\operatorname{Cap}_p(x,\Omega)^{-1}.
\]

For \(1<s<\infty\), we define
\[
\mathcal{R}_{p,s}(\Omega)
=
\max_{\substack{u\in X_V(\Omega)\\ u\not\equiv 0}}
\frac{\|u\|_{\ell^{ps'}(\Omega,\mu)}^p}{E_p(u)}.
\]
This maximum is well defined. Indeed, by homogeneity, it is enough to consider
the quotient on the set
\[
\{u\in X_V(\Omega):\ E_p(u)=1\}.
\]
Since \(X_V(\Omega)\) is finite-dimensional and \(E_p(\cdot)^{1/p}\) is a norm
on this space, this set is compact. Moreover, by the equivalence of norms in
finite dimension, the map
\[
X_V(\Omega)\ni u\mapsto \|u\|_{\ell^{ps'}(\Omega,\mu)}^p\in\mathbb{R}
\]
is continuous. Therefore, the maximum is attained.

\section{Statement of the results}\label{sec3}

Throughout this section, \((V,\omega,\mu)\) is a connected locally finite
weighted graph. We consider the elliptic problem
\begin{equation}\label{PP}
\left\{
\begin{array}{ll}
-\Delta_p^\omega u(x)=w(x)|u(x)|^{p-2}u(x), & x\in \Omega,\\[6pt]
u(x)=0, & x\in V\setminus\Omega,
\end{array}
\right.
\end{equation}
where \(1<p<\infty\), \(\Omega\subset V\) is a nonempty finite subset, and
\(w:\Omega\to\mathbb{R}\) is a potential.

By a solution of \eqref{PP}, we mean a function \(u\in X_V(\Omega)\) satisfying
\[
-\Delta_p^\omega u(x)=w(x)|u(x)|^{p-2}u(x),
\qquad x\in\Omega.
\]
A solution is said to be nontrivial if \(u\not\equiv 0\) in \(V\).

\subsection{Intrinsic Lyapunov-type inequalities}

The first main result of this paper is purely capacitary. It provides
intrinsic Lyapunov-type inequalities in terms of the capacitary radii
\(\mathcal{R}_p(\Omega)\) and \(\mathcal{R}_{p,s}(\Omega)\). These inequalities
hold on arbitrary connected locally finite weighted graphs and do not require
any metric, volume-growth, isoperimetric, or bounded-geometry assumptions.

\begin{theorem}\label{thm-cap-lyap}
Assume that \eqref{PP} admits a nontrivial solution
\(u\in X_V(\Omega)\). Then
\begin{equation}\label{Cap-Ly-ineq}
\frac{1}{\mathcal{R}_p(\Omega)}
\le
\|w_+\|_{\ell^1(\Omega,\mu)}.
\end{equation}
Moreover, for every \(1<s<\infty\),
\begin{equation}\label{Cap-Ly-ineq-s}
\frac{1}{\mathcal{R}_{p,s}(\Omega)}
\le
\|w_+\|_{\ell^s(\Omega,\mu)}.
\end{equation}
\end{theorem}

Although Theorem~\ref{thm-cap-lyap} is a direct consequence of the capacitary
definitions, it provides a useful reduction: once the capacitary radii
\(\mathcal{R}_p(\Omega)\) and \(\mathcal{R}_{p,s}(\Omega)\) are estimated, the
corresponding Lyapunov-type inequalities follow immediately. Next, we
illustrate this principle in several geometric situations.

\subsection{An inner-radius Lyapunov-type inequality}

We impose additional geometric assumptions to estimate
\(\mathcal{R}_p(\Omega)\) in terms of the inner radius \(r_\Omega\). More
precisely, we assume that \((V,\omega,\mu)\) satisfies the following conditions.

\begin{itemize}
\item[{\rm(G1)}] There exists a constant \(C_0\ge 1\) such that
\[
C_0^{-1}\le \mu(x)\le C_0,\qquad x\in V,
\]
\[
C_0^{-1}\le \omega_{xy}\le C_0,\qquad x\sim y,
\]
and
\[
\deg(x)\le C_0,\qquad x\in V.
\]

\item[{\rm(G2)}] There exist constants \(D>1\) and \(C_1\ge 1\) such that
\[
C_1^{-1}R^D\le \mu(B(x,R))\le C_1R^D,
\qquad x\in V,\quad R\ge 1.
\]

\item[{\rm(G3)}] There exists a constant \(C_2\ge 1\) such that, for every
\(x_0\in V\), every \(R\ge 1\), and every subset \(A\subset B(x_0,R)\),
\[
\min\{\mu(A),\mu(B(x_0,R)\setminus A)\}
\le C_2 R\,\omega(\partial_{B(x_0,R)}A),
\]
where
\[
\omega(\partial_{B(x_0,R)}A)
=
\sum_{\substack{x\in A,\ y\in B(x_0,R)\setminus A\\ y\sim x}}
\omega_{xy}.
\]
\end{itemize}

\begin{remark}\label{rem-assumptions}
Assumption {\rm(G1)} is a bounded-geometry condition. Assumption {\rm(G2)}
expresses two-sided polynomial volume growth; in particular, the
volume-growth exponent \(D\) is uniquely determined by this condition.
Assumption {\rm(G3)} is a relative isoperimetric inequality in balls.
\end{remark}

The preceding assumptions are satisfied by many graphs with polynomial volume
growth and controlled local geometry. In particular, they hold for the standard
integer lattice \(\mathbb{Z}^D\), endowed with counting measure and unit edge
weights. This example will serve as the basic model for the inner-radius
estimate below.

\begin{example}\label{ex-lattice}
Consider the standard lattice graph with vertex set \(\mathbb{Z}^D\), where
\(D\in\mathbb{N}\) and \(D\ge 2\). Two vertices \(x,y\in\mathbb{Z}^D\) are
adjacent, and we write \(x\sim y\), if and only if
\[
|x-y|_1=1,
\qquad
|x-y|_1=\sum_{i=1}^D |x_i-y_i|.
\]
We endow this graph with the counting measure
\[
\mu(x)=1,\qquad x\in\mathbb{Z}^D,
\]
and with unit edge weights
\[
\omega_{xy}=
\begin{cases}
1, & x\sim y,\\
0, & \text{otherwise}.
\end{cases}
\]
The graph \((\mathbb{Z}^D,\omega,\mu)\) is connected and locally finite, and
each vertex has exactly \(2D\) neighbours. The graph distance is given by
\[
d(x,y)=|x-y|_1,\qquad x,y\in\mathbb{Z}^D.
\]
Thus, for \(x\in\mathbb{Z}^D\) and \(R\ge 0\),
\[
B(x,R)
=
\{y\in\mathbb{Z}^D:\ |x-y|_1\le R\}
=
\{y\in\mathbb{Z}^D:\ |x-y|_1\le \lfloor R\rfloor\},
\]
where \(\lfloor R\rfloor\) denotes the integer part of \(R\).

Assumption {\rm(G1)} is immediate, since \(\mu\equiv 1\), \(\omega_{xy}=1\)
whenever \(x\sim y\), and \(\deg(x)=2D\) for all \(x\in\mathbb{Z}^D\).
Furthermore,
\[
\#B(x,R)\asymp R^D,\qquad x\in\mathbb{Z}^D,\quad R\ge 1,
\]
where the implicit constants depend only on \(D\). Thus, {\rm(G2)} is satisfied
with volume-growth exponent \(D\). Finally, the lattice \(\mathbb{Z}^D\)
satisfies the relative isoperimetric inequality in balls; see, for instance,
\cite[Assumption 1.1 and Remark 1.2]{AndresDeuschelSlowik2016}. The formulation
there is given in terms of the relative internal vertex boundary. On
\(\mathbb{Z}^D\), this formulation is equivalent to the edge-boundary
formulation in {\rm(G3)}, up to a multiplicative constant depending only on
\(D\). Hence, {\rm(G3)} holds.
\end{example}

Under the preceding assumptions, the capacitary radius admits the following
inner-radius estimate.

\begin{theorem}\label{thm-cap-radius-inner}
Assume that \((V,\omega,\mu)\) satisfies {\rm(G1)}--{\rm(G3)}, and let
\(p>D\), where \(D\) is the volume-growth exponent in {\rm(G2)}. Then there
exists a constant \(C=C(p,D,C_0,C_1,C_2)>0\) such that, for every nonempty
finite subset \(\Omega\subset V\),
\begin{equation}\label{Rp-G1-G3}
\mathcal{R}_p(\Omega)\le C r_\Omega^{p-D},
\end{equation}
where \(r_\Omega\) is the inner radius of \(\Omega\) defined by \eqref{r-Omega}.
\end{theorem}

The proof of Theorem~\ref{thm-cap-radius-inner} relies on two auxiliary
estimates. First, using the bounded-geometry condition {\rm(G1)} and the
relative isoperimetric assumption {\rm(G3)}, we derive a local Poincar\'e-type
inequality. Combining this estimate with the polynomial volume-growth condition
{\rm(G2)}, we obtain a discrete Morrey-type inequality for \(p>D\). This Morrey
estimate is the main analytic ingredient: when applied to admissible functions
in the definition of \(\operatorname{Cap}_p(x,\Omega)\), it gives uniform lower
bounds for the point capacities and hence an upper bound for the capacitary
radius.

\begin{remark}
The constant \(C\) in \eqref{Rp-G1-G3} depends only on
\(p,D,C_0,C_1,C_2\). In particular, it is uniform with respect to the finite
set \(\Omega\). The proof of Theorem~\ref{thm-cap-radius-inner} shows that such
a constant can be made explicit in terms of the constants appearing in
{\rm(G1)}--{\rm(G3)}. However, no optimality of this multiplicative constant is
claimed in general. The sharpness issue concerns instead the exponent
\(p-D\) in the power \(r_\Omega^{p-D}\). For specific classes of graphs,
sharper constants may be obtained by exploiting the precise geometry of the
graph.
\end{remark}

\begin{remark}
The restriction \(p>D\) reflects the use of a discrete Morrey-type estimate,
which provides pointwise control from the \(p\)-energy only in the supercritical
range. The borderline and subcritical cases \(p=D\) and \(1<p<D\) generally
require additional information on the geometry of the graph or more explicit
capacity estimates. This will be illustrated below for weighted radial path
graphs.
\end{remark}

An immediate consequence of Theorems~\ref{thm-cap-lyap} and
\ref{thm-cap-radius-inner} is the following inner-radius Lyapunov-type
inequality.

\begin{theorem}\label{th-inner-radius-lyap}
Assume that \((V,\omega,\mu)\) satisfies {\rm(G1)}--{\rm(G3)}, and let
\(p>D\), where \(D\) is the volume-growth exponent in {\rm(G2)}. Then there
exists a constant \(c=c(p,D,C_0,C_1,C_2)>0\) such that, for every nonempty
finite subset \(\Omega\subset V\), if \eqref{PP} admits a nontrivial solution
\(u\in X_V(\Omega)\), then
\begin{equation}\label{LY-IN1}
\|w_+\|_{\ell^1(\Omega,\mu)}
\ge
c r_\Omega^{-(p-D)}.
\end{equation}
\end{theorem}

\begin{remark}
Inequality \eqref{LY-IN1} can be viewed as a discrete counterpart of the
inner-radius Lyapunov-type inequality obtained by de N\'apoli and Pinasco
\cite{deNapoliPinasco2016} for the continuous \(p\)-Laplacian with
\(w\ge 0\) in the supercritical range \(p>N\). In the present graph setting,
the Euclidean dimension \(N\) is replaced by the volume-growth exponent \(D\),
the Euclidean inner radius by the graph inner radius \(r_\Omega\), and the
Lebesgue \(L^1\)-norm of the weight by the weighted
\(\ell^1(\Omega,\mu)\)-norm.
\end{remark}

Next, we show that the exponent \(p-D\) in \eqref{LY-IN1} cannot be improved
in general. This is proved in a uniform sense by considering a suitable family
of domains in the lattice \(\mathbb{Z}^D\), \(D\ge 2\), endowed with the graph
structure described in Example~\ref{ex-lattice}.

\begin{proposition}[Sharpness of the exponent \(p-D\)]\label{prop-sharp-p>D}
Let \(D\in\mathbb{N}\), \(D\ge 2\), and let \(p>D\). For every
\(\beta<p-D\) and every \(C_*>0\), there exists
\(R_0=R_0(\beta,C_*,p,D)>1\) such that, for every
\(R\in\mathbb{N}\) with \(R\ge R_0\), one can find a nonnegative potential
\(w_R:B(0,R)\to[0,\infty)\) and a nontrivial solution
\(u_R\in X_{\mathbb{Z}^D}(B(0,R))\) of
\[
\left\{
\begin{array}{ll}
-\Delta_p^\omega u_R(x)
=
w_R(x)|u_R(x)|^{p-2}u_R(x), & x\in B(0,R),\\[4pt]
u_R(x)=0, & x\in \mathbb{Z}^D\setminus B(0,R),
\end{array}
\right.
\]
such that
\begin{equation}\label{inv-ineq}
\|w_R\|_{\ell^1(B(0,R),\mu)}
<
\frac{C_*}{r_{B(0,R)}^{\beta}}.
\end{equation}
\end{proposition}

\begin{remark}
Proposition~\ref{prop-sharp-p>D} shows that the exponent \(p-D\) in
Theorem~\ref{th-inner-radius-lyap} is optimal: no estimate of the form
\[
\|w_+\|_{\ell^1(\Omega,\mu)}
\ge C r_\Omega^{-\beta},
\qquad \beta<p-D,
\]
with \(C>0\) independent of \(\Omega\), can hold uniformly for all finite
subsets of \(\mathbb{Z}^D\). The sharpness concerns only the exponent of
\(r_\Omega\), not the value of the constant. We also note that, in the graph
setting, \(r_\Omega\ge 1\) for every nonempty finite subset
\(\Omega\subset V\). In particular, for \(R\in\mathbb{N}\) and
\(B(0,R)\subset\mathbb{Z}^D\), one has
\[
r_{B(0,R)}=R+1.
\]
\end{remark}

The proof of Proposition~\ref{prop-sharp-p>D} is based on a concentration
argument. Using Lemma~\ref{lem-variational-localized-eigenvalue}, we construct
a nontrivial solution whose potential is supported in a smaller ball
\(A_R=B(0,R^\alpha)\). A suitable choice of \(\alpha\in(0,1)\) then gives the
required decay of \(\|w_R\|_{\ell^1(B(0,R),\mu)}\).

\subsection{Radial path graphs}\label{subs-radial}

The capacitary approach also applies beyond the supercritical range. To obtain
explicit capacitary estimates, we consider a weighted radial path graph with
effective dimension \(D>1\). In this model, the relevant scale is given by the
radial outer radius of the domain. The resulting Lyapunov-type inequalities
exhibit three different regimes: a power scale for \(p>D\), a logarithmic scale
for \(p=D\), and, for \(1<p<D\), a scale involving the \(\ell^s\)-norm of the
potential, where \(s>D/p\).

Let \(D>1\). Consider the vertex set
\[
V=\mathbb{N}_0=\{0,1,2,\ldots\}.
\]
Two vertices \(m,n\in V\) are adjacent, and we write \(m\sim n\), if and only if
\[
|m-n|=1.
\]
We define
\[
\mu(n)=(n+1)^{D-1},\qquad n\in\mathbb{N}_0,
\]
and
\[
\omega_{n,n+1}=\omega_{n+1,n}=(n+1)^{D-1},
\qquad n\in\mathbb{N}_0.
\]
All other edge weights are zero.

The resulting weighted graph \((V,\omega,\mu)\) is connected and locally
finite. Indeed, the graph is an infinite path starting from \(0\), and every
vertex has at most two neighbours. The graph distance is given by
\[
d(m,n)=|m-n|,\qquad m,n\in\mathbb{N}_0.
\]
Therefore, for \(R>0\),
\[
B(0,R)=\{0,1,\ldots,\lfloor R\rfloor\}.
\]
Moreover,
\[
\mu(B(0,R))
=
\sum_{n=0}^{\lfloor R\rfloor}(n+1)^{D-1}
\asymp R^D,\qquad R\ge 1.
\]
Thus, \(D\) plays the role of an effective dimension.

For a nonempty finite set \(\Omega\subset V\), we define its radial outer
radius by
\[
\tau_\Omega=\max\{n:\ n\in\Omega\}.
\]
Equivalently, \(\tau_\Omega\) is the smallest nonnegative integer \(k\) such
that
\[
\Omega\subset B(0,k).
\]

\begin{proposition}\label{prop-radial-path-cap-scales}
Let \(D>1\), and let \(\Omega\subset V\) be a nonempty finite subset such that
\(\tau_\Omega\ge 2\). Then the following estimates hold.
\begin{itemize}
\item[{\rm(i)}] If \(p>D\), then there exists a constant \(C=C(p,D)>0\) such that
\begin{equation}\label{est-cap1}
\mathcal{R}_p(\Omega)\le C \tau_\Omega^{p-D}.
\end{equation}

\item[{\rm(ii)}] If \(p=D\), then there exists a constant \(C=C(D)>0\) such that
\begin{equation}\label{est-cap2}
\mathcal{R}_p(\Omega)\le C(\log \tau_\Omega)^{p-1}.
\end{equation}

\item[{\rm(iii)}] If \(1<p<D\) and \(s>D/p\), then there exists a constant
\(C=C(p,D,s)>0\) such that
\begin{equation}\label{est-cap3}
\mathcal{R}_{p,s}(\Omega)\le C \tau_\Omega^{p-\frac{D}{s}}.
\end{equation}
\end{itemize}
\end{proposition}

\begin{remark}
The constants in Proposition~\ref{prop-radial-path-cap-scales} can be made
explicit by following the proof, since the estimates reduce to elementary
one-dimensional summations. We do not pursue optimal constants here; the
relevant point is the sharp order with respect to \(\tau_\Omega\) in the three
regimes.
\end{remark}

As a consequence of Theorem~\ref{thm-cap-lyap} and
Proposition~\ref{prop-radial-path-cap-scales}, we obtain the following
Lyapunov-type inequalities on radial path graphs.

\begin{theorem}\label{thm-radial-path-lyap}
Let \(D>1\), and let \(\Omega\subset V\) be a nonempty finite subset such that
\(\tau_\Omega\ge 2\). Assume that \eqref{PP} admits a nontrivial solution
\(u\in X_V(\Omega)\). Then the following estimates hold.
\begin{itemize}
\item[{\rm(i)}] If \(p>D\), then there exists a constant \(c=c(p,D)>0\) such
that
\[
\|w_+\|_{\ell^1(\Omega,\mu)}
\ge
c \tau_\Omega^{-(p-D)}.
\]

\item[{\rm(ii)}] If \(p=D\), then there exists a constant \(c=c(D)>0\) such
that
\[
\|w_+\|_{\ell^1(\Omega,\mu)}
\ge
c(\log \tau_\Omega)^{1-p}.
\]

\item[{\rm(iii)}] If \(1<p<D\) and \(s>D/p\), then there exists a constant
\(c=c(p,D,s)>0\) such that
\[
\|w_+\|_{\ell^s(\Omega,\mu)}
\ge
c \tau_\Omega^{-\left(p-\frac{D}{s}\right)}.
\]
\end{itemize}
\end{theorem}

\begin{remark}
The radial path graph illustrates a feature that is naturally captured by the
capacitary approach. Although the underlying graph is one-dimensional as a
combinatorial object, the choice of the vertex measure and edge weights
produces an effective dimension \(D\), which governs the capacitary and
Lyapunov scales. Thus, the relevant dimension is not the combinatorial
dimension of the graph, but the capacitary dimension encoded by
\((\omega,\mu)\). In this model, the estimates are expressed in terms of the
radial outer radius \(\tau_\Omega\). In particular, the graph displays the
three regimes \(p>D\), \(p=D\), and \(1<p<D\): a power scale, a logarithmic
scale, and an \(\ell^s\)-based scale, respectively. This contrasts with the
inner-radius estimate of Theorem~\ref{th-inner-radius-lyap}, which is based on
a supercritical Morrey-type mechanism and therefore applies only when \(p>D\).
\end{remark}

Next, we show that the three radial Lyapunov scales obtained above cannot be
improved with respect to their dependence on \(\tau_\Omega\).

\begin{proposition}[Sharpness of the radial Lyapunov scales]
\label{prop-radial-path-sharp}
Let \(D>1\). The estimates in Theorem~\ref{thm-radial-path-lyap} are sharp in
the following sense.
\begin{itemize}
\item[{\rm(i)}] If \(p>D\), then, for every \(\beta<p-D\) and every
\(C_*>0\), there exists \(R_0=R_0(\beta,C_*,p,D)>2\) such that, for every
\(R\in\mathbb{N}\) with \(R\ge R_0\), one can find a nonnegative potential
\(w_R:B(0,R)\to[0,\infty)\) and a nontrivial solution
\(u_R\in X_V(B(0,R))\) of \eqref{PP}, with \(\Omega=B(0,R)\) and \(w=w_R\),
such that
\begin{equation}\label{sharp-radial-1}
\|w_R\|_{\ell^1(B(0,R),\mu)}
<
C_* R^{-\beta}.
\end{equation}

\item[{\rm(ii)}] If \(p=D\), then, for every \(\beta<p-1\) and every
\(C_*>0\), there exists \(R_0=R_0(\beta,C_*,D)>2\) such that, for every
\(R\in\mathbb{N}\) with \(R\ge R_0\), one can find a nonnegative potential
\(w_R:B(0,R)\to[0,\infty)\) and a nontrivial solution
\(u_R\in X_V(B(0,R))\) of \eqref{PP}, with \(\Omega=B(0,R)\) and \(w=w_R\),
such that
\begin{equation}\label{sharp-radial-2}
\|w_R\|_{\ell^1(B(0,R),\mu)}
<
C_* (\log R)^{-\beta}.
\end{equation}

\item[{\rm(iii)}] If \(1<p<D\) and \(s>D/p\), then, for every
\(\beta<p-\frac{D}{s}\) and every \(C_*>0\), there exists
\(R_0=R_0(\beta,C_*,p,D,s)>2\) such that, for every
\(R\in\mathbb{N}\) with \(R\ge R_0\), one can find a nonnegative potential
\(w_R:B(0,R)\to[0,\infty)\) and a nontrivial solution
\(u_R\in X_V(B(0,R))\) of \eqref{PP}, with \(\Omega=B(0,R)\) and \(w=w_R\),
such that
\begin{equation}\label{sharp-radial-3}
\|w_R\|_{\ell^s(B(0,R),\mu)}
<
C_* R^{-\beta}.
\end{equation}
\end{itemize}
\end{proposition}

The proof of Proposition~\ref{prop-radial-path-sharp} relies on
Lemma~\ref{lem-variational-localized-eigenvalue} and on suitable
one-dimensional test functions. In the regimes \(p>D\) and \(1<p<D\), the
constructed potentials are supported on the whole ball \(\Omega_R=B(0,R)\),
whereas in the critical case \(p=D\), the potential is concentrated at the
origin. These choices yield, respectively, the power, logarithmic, and
subcritical scales appearing in \eqref{sharp-radial-1}--\eqref{sharp-radial-3}.

\subsection{Regular trees}\label{subs-trees}

Let \(q\ge 3\) be an integer. We denote by \(\mathbb{T}_q=(V,E)\) the
\(q\)-regular tree, that is, the infinite connected graph in which every vertex
has degree \(q\) and which contains no cycles. Two vertices \(x,y\in V\) are
adjacent, and we write \(x\sim y\), if and only if \(\{x,y\}\in E\).

We endow \(\mathbb{T}_q\) with the counting measure
\[
\mu(x)=1,\qquad x\in V,
\]
and unit edge weights
\[
\omega_{xy}=
\begin{cases}
1, & x\sim y,\\
0, & \text{otherwise}.
\end{cases}
\]
The resulting weighted graph \((V,\omega,\mu)\) is connected and locally
finite.

Regular trees provide a contrasting example to the polynomial-growth graphs
considered above. In this case, the capacitary scale is uniform over arbitrary
finite domains and can be computed explicitly.

\begin{proposition}\label{prop-tree-cap}
Let \(\Omega\subset V\) be a nonempty finite subset. Set
\[
c_{p,q}
=
q\left(1-(q-1)^{-\frac{1}{p-1}}\right)^{p-1}.
\]
Then
\begin{equation}\label{cap-RT}
\mathcal{R}_p(\Omega)\le \frac{1}{c_{p,q}}.
\end{equation}
\end{proposition}

As a consequence of Theorem~\ref{thm-cap-lyap} and
Proposition~\ref{prop-tree-cap}, we obtain the following uniform
Lyapunov-type inequality on regular trees.

\begin{theorem}\label{thm-tree-lyap}
Let \(\Omega\subset V\) be a nonempty finite subset. Assume that \eqref{PP}
admits a nontrivial solution \(u\in X_V(\Omega)\). Then
\begin{equation}\label{tree-lyap}
\|w_+\|_{\ell^1(\Omega,\mu)}
\ge c_{p,q}.
\end{equation}
\end{theorem}

Finally, we show that the constant \(c_{p,q}\) in the uniform Lyapunov-type
inequality \eqref{tree-lyap} cannot be improved. To this end, we fix a vertex
\(o\in V\) and consider the balls \(B(o,R)\), \(R\in\mathbb{N}\).

\begin{proposition}[Sharpness of the uniform constant]
\label{prop-tree-sharp}
The constant \(c_{p,q}\) in Theorem~\ref{thm-tree-lyap} is optimal in the
following sense. For every \(C_*>c_{p,q}\), there exists
\(R_0=R_0(C_*,p,q)>1\) such that, for every \(R\in\mathbb{N}\) with
\(R\ge R_0\), one can find a nonnegative potential
\(w_R:B(o,R)\to[0,\infty)\) and a nontrivial solution
\(u_R\in X_V(B(o,R))\) of \eqref{PP}, with \(\Omega=B(o,R)\) and \(w=w_R\),
such that
\[
\|w_R\|_{\ell^1(B(o,R),\mu)}<C_*.
\]
\end{proposition}

The proof of Proposition~\ref{prop-tree-sharp} relies on capacitary minimizers
for the point capacity of a fixed vertex in large balls. These minimizers yield
solutions of \eqref{PP} with potentials concentrated at that vertex. The
explicit capacity asymptotics then show that the constant \(c_{p,q}\) is
approached along this family of balls and hence cannot be improved.

\section{Some applications to weighted eigenvalue problems}\label{sec4}

Spectral problems for discrete \(p\)-Laplacians on graphs have been studied
from several viewpoints. Takeuchi~\cite{Takeuchi2003} investigated the
spectrum of the \(p\)-Laplacian on infinite graphs and obtained Cheeger- and
Brooks-type estimates. Amghibech~\cite{Amghibech2003} studied eigenvalues of
the discrete \(p\)-Laplacian on finite graphs, including weighted
formulations. Cheeger-type inequalities and spectral estimates for
\(p\)-Laplacians on weighted graphs were obtained by Keller and
Mugnolo~\cite{KellerMugnolo2016}. Dirichlet \(p\)-Laplacian eigenvalues and
their connections with Cheeger constants were also investigated by Hua and
Wang~\cite{HuaWang2020}. More recently, Lin, Liu, You, and
Zhao~\cite{LinLiuYouZhao2025} obtained lower bounds for the first Dirichlet
eigenvalue on graphs in terms of the inscribed radius and the cardinality of
the interior. Faber--Krahn-type inequalities for the first Dirichlet eigenvalue
of the combinatorial \(p\)-Laplacian on graphs with boundary were studied by
He and Yu~\cite{HeYu2026}.

In this section, we apply the Lyapunov-type inequalities established in the
preceding sections to weighted Dirichlet eigenvalue problems. More precisely,
for a nontrivial weight \(w\ge 0\), we obtain capacitary lower bounds for the
first weighted Dirichlet eigenvalue \(\lambda_{1,p}(\Omega,w)\). These bounds
depend explicitly on \(\ell^s\)-norms of \(w\) and complement Cheeger-type and
Faber--Krahn-type estimates. We then translate them into explicit eigenvalue
bounds in the geometric settings considered above.

Let \((V,\omega,\mu)\) be a connected locally finite weighted graph, and let
\(\Omega\subset V\) be a nonempty finite subset. For a weight
\(w:\Omega\to[0,\infty)\) with \(w\not\equiv0\) and for \(1<p<\infty\), we
define the first weighted Dirichlet eigenvalue by
\[
\lambda_{1,p}(\Omega,w)
=
\inf_{\substack{u\in X_V(\Omega)\\
\sum_{x\in\Omega}\mu(x)w(x)|u(x)|^p>0}}
\frac{E_p(u)}
{\displaystyle\sum_{x\in\Omega}\mu(x)w(x)|u(x)|^p}.
\]
By the same finite-dimensional variational argument as in
Lemma~\ref{lem-variational-localized-eigenvalue}, the infimum is attained by
some nontrivial \(u_1\in X_V(\Omega)\). Moreover, \(u_1\) satisfies
\[
-\Delta_p^\omega u_1(x)
=
\lambda_{1,p}(\Omega,w)w(x)|u_1(x)|^{p-2}u_1(x),
\qquad x\in\Omega.
\]
Thus \(u_1\) is a nontrivial solution of \eqref{PP} in \(\Omega\), with
potential
\[
x\mapsto \lambda_{1,p}(\Omega,w)w(x).
\]

Applying the Lyapunov-type inequalities of Theorem~\ref{thm-cap-lyap}, we
obtain the following capacitary lower bounds for the first weighted Dirichlet
eigenvalue.

\begin{corollary}\label{cor-weighted-eigen-cap}
Under the preceding assumptions, one has
\[
\lambda_{1,p}(\Omega,w)
\ge
\frac{1}{\mathcal{R}_p(\Omega)\|w\|_{\ell^1(\Omega,\mu)}}.
\]
Moreover, for every \(1<s<\infty\),
\[
\lambda_{1,p}(\Omega,w)
\ge
\frac{1}{\mathcal{R}_{p,s}(\Omega)\|w\|_{\ell^s(\Omega,\mu)}}.
\]
\end{corollary}

Next, we apply Theorem~\ref{th-inner-radius-lyap} to the eigenvalue problem on
graphs satisfying {\rm(G1)}--{\rm(G3)}. This gives the following lower bound
for \(\lambda_{1,p}(\Omega,w)\) in terms of the inner radius.

\begin{corollary}\label{cor-eigen-inner-radius}
Assume in addition that \((V,\omega,\mu)\) satisfies
{\rm(G1)}--{\rm(G3)} and that \(p>D\), where \(D\) is the volume-growth
exponent in {\rm(G2)}. Then there exists a constant
\(c=c(p,D,C_0,C_1,C_2)>0\) such that
\begin{equation}\label{eig-inner-radius}
\lambda_{1,p}(\Omega,w)
\ge
\frac{c}{r_\Omega^{p-D}\|w\|_{\ell^1(\Omega,\mu)}}.
\end{equation}
\end{corollary}

\begin{remark}
Estimate \eqref{eig-inner-radius} can be viewed as a discrete analogue of the
inner-radius eigenvalue bound obtained from the Lyapunov-type inequality of
de N\'apoli and Pinasco~\cite{deNapoliPinasco2016} for the continuous
\(p\)-Laplacian in the supercritical regime \(p>N\).
\end{remark}

For the radial path graph introduced in Subsection~\ref{subs-radial},
Theorem~\ref{thm-radial-path-lyap} yields the following estimates.

\begin{corollary}\label{cor-eigen-radial-path}
Assume that \(\tau_\Omega\ge 2\). Then the following estimates hold.
\begin{itemize}
\item[{\rm(i)}] If \(p>D\), then there exists a constant \(c=c(p,D)>0\) such
that
\[
\lambda_{1,p}(\Omega,w)
\ge
\frac{c}{\tau_\Omega^{p-D}\|w\|_{\ell^1(\Omega,\mu)}}.
\]

\item[{\rm(ii)}] If \(p=D\), then there exists a constant \(c=c(D)>0\) such
that
\[
\lambda_{1,p}(\Omega,w)
\ge
\frac{c}{(\log \tau_\Omega)^{p-1}\|w\|_{\ell^1(\Omega,\mu)}}.
\]

\item[{\rm(iii)}] If \(1<p<D\) and \(s>D/p\), then there exists a constant
\(c=c(p,D,s)>0\) such that
\[
\lambda_{1,p}(\Omega,w)
\ge
\frac{c}{\tau_\Omega^{p-\frac{D}{s}}\|w\|_{\ell^s(\Omega,\mu)}}.
\]
\end{itemize}
\end{corollary}

Finally, we consider the regular tree introduced in
Subsection~\ref{subs-trees}. Theorem~\ref{thm-tree-lyap} yields the following
uniform lower bound.

\begin{corollary}\label{cor-eigen-tree}
Let \(\mathbb{T}_q\), \(q\ge 3\), be the \(q\)-regular tree endowed with
counting measure and unit edge weights. Then
\[
\lambda_{1,p}(\Omega,w)
\ge
\frac{c_{p,q}}{\|w\|_{\ell^1(\Omega,\mu)}},
\]
where
\[
c_{p,q}
=
q\left(1-(q-1)^{-\frac{1}{p-1}}\right)^{p-1}.
\]
\end{corollary}

\begin{remark}
The sharpness results stated in Section~\ref{sec3} also transfer to the
eigenvalue estimates in a weighted sense. More precisely, if the weight is
allowed to depend on the domain, then the exponents of \(r_\Omega\) and
\(\tau_\Omega\), the logarithmic order in the critical radial case, and the
constant \(c_{p,q}\) in the tree case cannot be improved uniformly. This does
not necessarily imply optimality for a fixed weight, such as \(w\equiv 1\).
\end{remark}

\section{Proofs of the main results}\label{sec5}

In this section, we prove the results stated in Section~\ref{sec3}.
Throughout the proofs, \(C>0\) denotes a generic constant whose value may
change from line to line but remains independent of the relevant scaling
parameters.

\subsection{Lyapunov-type inequalities via capacitary radii}

\begin{proof}[Proof of Theorem~\ref{thm-cap-lyap}]
Let \(u\in X_V(\Omega)\) be a nontrivial solution of \eqref{PP}. Choose
\(x_0\in\Omega\) such that
\[
|u(x_0)|=\|u\|_{\ell^\infty(\Omega)}.
\]
Since \(u\not\equiv 0\), we have \(u(x_0)\neq 0\).

Multiplying \eqref{PP} by \(u\), summing over \(\Omega\), and using the
discrete integration-by-parts identity, we obtain
\begin{equation}\label{cap-proof-energy-id}
E_p(u)=\sum_{x\in\Omega}\mu(x)w(x)|u(x)|^p.
\end{equation}
Since \(w\le w_+\) in \(\Omega\), it follows from
\eqref{cap-proof-energy-id} that
\begin{equation}\label{nsit}
E_p(u)
\le
\sum_{x\in\Omega}\mu(x)w_+(x)|u(x)|^p.
\end{equation}
Hence
\begin{equation}\label{cap-proof-l1-bound}
\frac{E_p(u)}{\|u\|_{\ell^\infty(\Omega)}^p}
\le
\|w_+\|_{\ell^1(\Omega,\mu)}.
\end{equation}

Set
\[
\varphi=\frac{u}{u(x_0)}.
\]
Then, \(\varphi\in X_V(\Omega)\) and \(\varphi(x_0)=1\). Therefore, by the
definition of \(\operatorname{Cap}_p(x_0,\Omega)\),
\begin{equation}\label{cap-proof-point-cap}
\operatorname{Cap}_p(x_0,\Omega)
\le
E_p(\varphi)
=
\frac{E_p(u)}{|u(x_0)|^p}
=
\frac{E_p(u)}{\|u\|_{\ell^\infty(\Omega)}^p}.
\end{equation}
On the other hand, by the definition of \(\mathcal{R}_p(\Omega)\),
\begin{equation}\label{cap-proof-radius}
\frac{1}{\mathcal{R}_p(\Omega)}
=
\min_{x\in\Omega}\operatorname{Cap}_p(x,\Omega)
\le
\operatorname{Cap}_p(x_0,\Omega).
\end{equation}
Combining \eqref{cap-proof-l1-bound}, \eqref{cap-proof-point-cap}, and
\eqref{cap-proof-radius}, we obtain
\[
\frac{1}{\mathcal{R}_p(\Omega)}
\le
\frac{E_p(u)}{\|u\|_{\ell^\infty(\Omega)}^p}
\le
\|w_+\|_{\ell^1(\Omega,\mu)}.
\]
This proves \eqref{Cap-Ly-ineq}.

Fix \(1<s<\infty\). From \eqref{nsit} and H\"older's inequality, we obtain
\[
E_p(u)
\le
\|w_+\|_{\ell^s(\Omega,\mu)}
\|u\|_{\ell^{ps'}(\Omega,\mu)}^p.
\]
Since \(u\not\equiv 0\), we have
\[
\|u\|_{\ell^{ps'}(\Omega,\mu)}>0.
\]
Hence
\begin{equation}\label{div-nom-u}
\frac{E_p(u)}{\|u\|_{\ell^{ps'}(\Omega,\mu)}^p}
\le
\|w_+\|_{\ell^s(\Omega,\mu)}.
\end{equation}
By the definition of \(\mathcal{R}_{p,s}(\Omega)\),
\begin{equation}\label{cap-proof-Rps}
\frac{1}{\mathcal{R}_{p,s}(\Omega)}
\le
\frac{E_p(u)}
{\|u\|_{\ell^{ps'}(\Omega,\mu)}^p}.
\end{equation}
Combining \eqref{div-nom-u} and \eqref{cap-proof-Rps}, we obtain
\[
\frac{1}{\mathcal{R}_{p,s}(\Omega)}
\le
\|w_+\|_{\ell^s(\Omega,\mu)}.
\]
This proves \eqref{Cap-Ly-ineq-s} and completes the proof.
\end{proof}

\subsection{Capacitary estimates under \({\rm(G1)}\)--\({\rm(G3)}\)}

In this subsection, we prove Theorem~\ref{thm-cap-radius-inner} and
Proposition~\ref{prop-sharp-p>D}. First, we establish two auxiliary estimates
needed in the proof of Theorem~\ref{thm-cap-radius-inner}.

\begin{lemma}[Local Poincar\'e inequality]\label{lem-local-poincare}
Assume that \((V,\omega,\mu)\) satisfies {\rm(G1)} and {\rm(G3)}. Then, for
every \(1<p<\infty\), there exists a constant \(C_P=C_P(p,C_0,C_2)>0\)
such that, for every \(x_0\in V\), every \(R\ge1\), and every
\(u:V\to\mathbb R\),
\begin{equation}\label{local-poincare}
\sum_{x\in B(x_0,R)}\mu(x)|u(x)-u_{B(x_0,R)}|^p
\le
C_P R^p
\sum_{\substack{x,y\in B(x_0,R)\\ y\sim x}}
\omega_{xy}|u(y)-u(x)|^p.
\end{equation}
\end{lemma}

\begin{proof}

Throughout the proof, \(C>0\) denotes a generic constant depending only on \(p,C_0,C_2\).

Let \(x_0\in V\), \(R\ge1\), and \(u:V\to\mathbb R\). Set
\[
B_R=B(x_0,R).
\]
Since the graph is locally finite, \(B_R\) is finite. We may therefore choose
a median \(m\) of \(u\) in \(B_R\), that is,
\[
\mu(\{x\in B_R:\ u(x)>m\})\le \frac{\mu(B_R)}{2},
\qquad
\mu(\{x\in B_R:\ u(x)<m\})\le \frac{\mu(B_R)}{2}.
\]

First, we  prove a weighted \(\ell^p\)-estimate for the positive part \((u-m)_+\). Set
\[
f=(u-m)_+,
\qquad
A_t=\{x\in B_R:\ f(x)>t\},\quad t>0.
\]
Since
\[
A_t\subset \{x\in B_R:\ u(x)>m\},
\]
the median property gives \(\mu(A_t)\le \mu(B_R)/2\). Hence
\[
\min\{\mu(A_t),\mu(B_R\setminus A_t)\}=\mu(A_t).
\]
By {\rm(G3)}, we obtain
\[
\mu(A_t)\le C_2R\,\omega(\partial_{B_R}A_t).
\]
Using the layer-cake formula and the definition of
\(\omega(\partial_{B_R}A_t)\), we get
\[
\begin{aligned}
\sum_{x\in B_R}\mu(x)f(x)^p
&=
p\int_0^\infty t^{p-1}\mu(A_t)\,dt \\
&\le
C_2R
\int_0^\infty p t^{p-1}\omega(\partial_{B_R}A_t)\,dt  \\
&=
C_2R
\sum_{\substack{x,y\in B_R\\ y\sim x}}
\omega_{xy}
\int_0^\infty p t^{p-1}
\mathbf 1_{\{f(y)\le t<f(x)\}}\,dt \\
&\le
C_2R
\sum_{\substack{x,y\in B_R\\ y\sim x}}
\omega_{xy}|f(x)^p-f(y)^p|.
\end{aligned}
\]
Using
\[
|a^p-b^p|
\le
p(a^{p-1}+b^{p-1})|a-b|,
\qquad a,b\ge0,
\]
and H\"older's inequality, we obtain
\[
\sum_{x\in B_R}\mu(x)f(x)^p
\le
C R
\left(
\sum_{\substack{x,y\in B_R\\ y\sim x}}
\omega_{xy}|f(y)-f(x)|^p
\right)^{1/p} 
\left(
\sum_{\substack{x,y\in B_R\\ y\sim x}}
\omega_{xy}
\bigl(f(x)^{p-1}+f(y)^{p-1}\bigr)^{\frac{p}{p-1}}
\right)^{\frac{p-1}{p}}.
\]
Moreover,
\[
\bigl(f(x)^{p-1}+f(y)^{p-1}\bigr)^{\frac{p}{p-1}}
\le
C\bigl(f(x)^p+f(y)^p\bigr).
\]
Using {\rm(G1)}, namely the boundedness of the edge weights and degrees
together with the lower bound on \(\mu\), we obtain
\[
\sum_{\substack{x,y\in B_R\\ y\sim x}}
\omega_{xy}\bigl(f(x)^p+f(y)^p\bigr)
\le
C
\sum_{x\in B_R}\mu(x)f(x)^p.
\]
Therefore,
\[
\sum_{x\in B_R}\mu(x)f(x)^p
\le
C R
\left(
\sum_{\substack{x,y\in B_R\\ y\sim x}}
\omega_{xy}|f(y)-f(x)|^p
\right)^{1/p}
\left(
\sum_{x\in B_R}\mu(x)f(x)^p
\right)^{\frac{p-1}{p}}.
\]
Since \(s\mapsto (s-m)_+\) is \(1\)-Lipschitz, we have
\[
|f(y)-f(x)|\le |u(y)-u(x)|,
\qquad x,y\in B_R.
\]
Consequently,
\begin{equation}\label{chang1}
\sum_{x\in B_R}\mu(x)f(x)^p
\le
C R^p
\sum_{\substack{x,y\in B_R\\ y\sim x}}
\omega_{xy}|u(y)-u(x)|^p.
\end{equation}

Applying the same argument to \(g=(m-u)_+\), we obtain
\begin{equation}\label{chang2}
\sum_{x\in B_R}\mu(x)g(x)^p
\le
C R^p
\sum_{\substack{x,y\in B_R\\ y\sim x}}
\omega_{xy}|u(y)-u(x)|^p.
\end{equation}
Since \(|u-m|^p=f^p+g^p\), it follows from \eqref{chang1} and \eqref{chang2} that
\begin{equation}\label{chang3}
\sum_{x\in B_R}\mu(x)|u(x)-m|^p
\le
C R^p
\sum_{\substack{x,y\in B_R\\ y\sim x}}
\omega_{xy}|u(y)-u(x)|^p.
\end{equation}

It remains to pass from the median to the average. By Jensen's inequality,
\[
|u_{B_R}-m|^p
\le
\frac{1}{\mu(B_R)}
\sum_{x\in B_R}\mu(x)|u(x)-m|^p.
\]
Hence
\begin{equation}\label{chang4}
\begin{aligned}
\sum_{x\in B_R}\mu(x)|u(x)-u_{B_R}|^p
&\le
2^{p-1}\sum_{x\in B_R}\mu(x)|u(x)-m|^p
+
2^{p-1}\mu(B_R)|u_{B_R}-m|^p  \\
&\le
2^p
\sum_{x\in B_R}\mu(x)|u(x)-m|^p.
\end{aligned}
\end{equation}
Combining \eqref{chang3} and \eqref{chang4}, we
obtain \eqref{local-poincare} for some \(C_P=C_P(p,C_0,C_2)>0\).
\end{proof}

\begin{lemma}[Local Morrey-type estimate]\label{lem-local-morrey}
Assume that \((V,\omega,\mu)\) satisfies {\rm(G1)}--{\rm(G3)}. Let
\(p>D\), where \(D\) is the volume-growth exponent in {\rm(G2)}. Then there
exists a constant
\[
C_M=C_M(p,D,C_0,C_1,C_2)>0
\]
such that, for every \(x_0\in V\), every \(R\ge 1\), and every
\(u:V\to\mathbb{R}\),
\begin{equation}\label{Morrey-ineq}
\sup_{x,y\in B(x_0,R)} |u(x)-u(y)|
\le
C_M R^{1-\frac{D}{p}}
\left(
\sum_{\substack{z,z'\in B(x_0,2R)\\ z'\sim z}}
\omega_{zz'}|u(z')-u(z)|^p
\right)^{1/p}.
\end{equation}
\end{lemma}

\begin{proof}
Throughout the proof, \(C>0\) denotes a generic constant depending only on
\(p,D,C_0,C_1,C_2\).

Let \(x_0\in V\), \(R\ge 1\), and \(u:V\to\mathbb{R}\). Set
\[
\mathcal{E}_R(u)
=
\sum_{\substack{x,y\in B(x_0,2R)\\ y\sim x}}
\omega_{xy}|u(y)-u(x)|^p.
\]

First, we prove that, for every \(z\in B(x_0,R)\),
\begin{equation}\label{morrey-point-estimate}
|u(z)-u_{B(z,R)}|
\le
C R^{1-\frac{D}{p}}\mathcal{E}_R(u)^{1/p}.
\end{equation}
Fix \(z\in B(x_0,R)\). Let
\[
N=\lfloor \log_2 R\rfloor.
\]
Then \(N\ge 0\) and
\[
1\le 2^{-N}R<2.
\]
For \(k=0,\ldots,N\), set
\[
\rho_k=2^{-k}R,
\qquad
B_k=B(z,\rho_k).
\]
If \(N\ge 1\), then
\[
B_{k+1}\subset B_k,\qquad k=0,\ldots,N-1.
\]
Moreover, since \(z\in B(x_0,R)\) and \(\rho_k\le R\), we have
\[
B_k\subset B(x_0,2R),
\qquad k=0,\ldots,N.
\]
For \(k=0,\ldots,N-1\), when this set of indices is nonempty, H\"older's
inequality, the volume lower bound in {\rm(G2)}, and the local Poincar\'e
inequality \eqref{local-poincare} give
\[
\begin{aligned}
|u_{B_{k+1}}-u_{B_k}|
&\le
\mu(B_{k+1})^{-1/p}
\left(
\sum_{x\in B_k}\mu(x)|u(x)-u_{B_k}|^p
\right)^{1/p}  \\
&\le
C\rho_{k+1}^{-\frac{D}{p}}\rho_k
\left(
\sum_{\substack{x,y\in B_k\\ y\sim x}}
\omega_{xy}|u(y)-u(x)|^p
\right)^{1/p}.
\end{aligned}
\]
Since \(\rho_{k+1}=\rho_k/2\) and \(B_k\subset B(x_0,2R)\), it follows that
\[
|u_{B_{k+1}}-u_{B_k}|
\le
C\rho_k^{1-\frac{D}{p}}\mathcal{E}_R(u)^{1/p}.
\]
Hence, because \(p>D\),
\begin{equation}\label{est-diff-u}
\sum_{k=0}^{N-1}|u_{B_{k+1}}-u_{B_k}|
\le
C R^{1-\frac{D}{p}}\mathcal{E}_R(u)^{1/p}.
\end{equation}
When \(N=0\), the sum on the left-hand side is empty, and the estimate is
understood to hold trivially.

Next, we estimate \(|u(z)-u_{B_N}|\). Since \(z\in B_N\), by {\rm(G1)} and
\eqref{local-poincare},
\[
\begin{aligned}
|u(z)-u_{B_N}|
&\le
C
\left(
\sum_{x\in B_N}\mu(x)|u(x)-u_{B_N}|^p
\right)^{1/p}  \\
&\le
C\rho_N
\left(
\sum_{\substack{x,y\in B_N\\ y\sim x}}
\omega_{xy}|u(y)-u(x)|^p
\right)^{1/p}.
\end{aligned}
\]
Using \(B_N\subset B(x_0,2R)\), \(1\le \rho_N<2\), and
\(R^{1-D/p}\ge 1\), we obtain
\begin{equation}\label{est-diff-u2}
|u(z)-u_{B_N}|
\le
C R^{1-\frac{D}{p}}\mathcal{E}_R(u)^{1/p}.
\end{equation}
Since \(B_0=B(z,R)\), the telescopic identity
\[
u(z)-u_{B(z,R)}
=
u(z)-u_{B_N}
+
\sum_{k=0}^{N-1}\bigl(u_{B_{k+1}}-u_{B_k}\bigr)
\]
together with \eqref{est-diff-u} and \eqref{est-diff-u2} gives
\eqref{morrey-point-estimate}.

Let \(x,y\in B(x_0,R)\), and set
\[
B^*=B(x_0,2R).
\]
Then \(B(x,R)\subset B^*\) and \(B(y,R)\subset B^*\). We write
\[
|u(x)-u(y)|
\le
|u(x)-u_{B(x,R)}|
+
|u_{B(x,R)}-u_{B^*}|
+
|u_{B^*}-u_{B(y,R)}|
+
|u_{B(y,R)}-u(y)|.
\]
The first and last terms are bounded by \eqref{morrey-point-estimate}. For the
middle terms, H\"older's inequality, the lower volume bound in {\rm(G2)}, and
\eqref{local-poincare} applied to \(B^*\) give
\[
\begin{aligned}
|u_{B(x,R)}-u_{B^*}|
&\le
\mu(B(x,R))^{-1/p}
\left(
\sum_{\xi\in B^*}\mu(\xi)|u(\xi)-u_{B^*}|^p
\right)^{1/p}  \\
&\le
C R^{-\frac{D}{p}}
\cdot R
\left(
\sum_{\substack{\xi,\eta\in B^*\\ \eta\sim \xi}}
\omega_{\xi\eta}|u(\eta)-u(\xi)|^p
\right)^{1/p}  \\
&\le
C R^{1-\frac{D}{p}}\mathcal{E}_R(u)^{1/p}.
\end{aligned}
\]
The same estimate holds for \(|u_{B(y,R)}-u_{B^*}|\). Therefore,
\[
|u(x)-u(y)|
\le
C R^{1-\frac{D}{p}}\mathcal{E}_R(u)^{1/p}.
\]
Taking the supremum over \(x,y\in B(x_0,R)\), we obtain
\eqref{Morrey-ineq} for some
\[
C_M=C_M(p,D,C_0,C_1,C_2)>0.
\]
\end{proof}

We are now ready to prove Theorem~\ref{thm-cap-radius-inner}.

\begin{proof}[Proof of Theorem~\ref{thm-cap-radius-inner}]
Throughout the proof, \(C>0\) denotes a generic constant depending only on
\(p,D,C_0,C_1,C_2\).

Let \(\Omega\subset V\) be a nonempty finite set. Since \(V\) is infinite,
\(V\setminus\Omega\neq\emptyset\), and by the notation introduced above,
\[
1\le r_\Omega<\infty.
\]
Fix \(x\in\Omega\). Since the graph distance takes values in
\(\mathbb N_0\), the infimum defining \(d(x,V\setminus\Omega)\) is attained.
Hence there exists \(y\in V\setminus\Omega\) such that
\[
d(x,y)=d(x,V\setminus\Omega)\le r_\Omega.
\]

Let \(\varphi\in X_V(\Omega)\) be such that \(\varphi(x)=1\). Since
\(\varphi=0\) in \(V\setminus\Omega\), we have \(\varphi(y)=0\). Hence
\[
1=|\varphi(x)-\varphi(y)|.
\]
Moreover, since \(d(x,y)\le r_\Omega\), both \(x\) and \(y\) belong to
\(B(x,r_\Omega)\). Applying the local Morrey-type estimate
\eqref{Morrey-ineq} with center \(x\) and radius \(r_\Omega\), we get
\[
1
\le
C r_\Omega^{1-\frac{D}{p}}
\left(
\sum_{\substack{z,z'\in B(x,2r_\Omega)\\ z'\sim z}}
\omega_{zz'}|\varphi(z')-\varphi(z)|^p
\right)^{1/p}.
\]
On the other hand, by the definition of \(E_p\),
\[
\sum_{\substack{z,z'\in B(x,2r_\Omega)\\ z'\sim z}}
\omega_{zz'}|\varphi(z')-\varphi(z)|^p
\le
2E_p(\varphi).
\]
Therefore,
\[
1
\le
C r_\Omega^{1-\frac{D}{p}} E_p(\varphi)^{1/p}.
\]
Equivalently, after changing \(C\),
\[
E_p(\varphi)
\ge
C r_\Omega^{-(p-D)}.
\]
Taking the infimum over all \(\varphi\in X_V(\Omega)\) such that
\(\varphi(x)=1\), we obtain
\[
\operatorname{Cap}_p(x,\Omega)
\ge
C r_\Omega^{-(p-D)}.
\]
Since this estimate holds for every \(x\in\Omega\), it follows that
\[
\min_{x\in\Omega}\operatorname{Cap}_p(x,\Omega)
\ge
C r_\Omega^{-(p-D)}.
\]
Using the definition
\[
\mathcal R_p(\Omega)
=
\max_{x\in\Omega}\operatorname{Cap}_p(x,\Omega)^{-1},
\]
we conclude, after changing \(C\) again, that
\[
\mathcal R_p(\Omega)
\le
Cr_\Omega^{p-D}.
\]
This proves \eqref{Rp-G1-G3}.
\end{proof}

\begin{proof}[Proof of Proposition~\ref{prop-sharp-p>D}]
Throughout the proof, \(C>0\) denotes a generic constant independent of \(R\).

Let \(\beta<p-D\) and \(C_*>0\) be fixed. Choose \(\alpha\in(0,1)\) such that
\[
\frac{\beta}{p-D}<\alpha<1.
\]
For \(R\in\mathbb{N}\) sufficiently large, set
\[
\Omega_R=B(0,R),
\qquad
\rho_R=R^\alpha,
\qquad
A_R=B(0,\rho_R).
\]
Since \(0<\alpha<1\), we have \(1<\rho_R<R\) for \(R\) large. Hence
\(A_R\subset\Omega_R\). Moreover, by the volume growth of balls in
\(\mathbb{Z}^D\), there exists a constant \(C>0\), independent of \(R\), such
that
\begin{equation}\label{sharp-volume-AR}
C^{-1}\rho_R^D\le \#A_R\le C\rho_R^D.
\end{equation}

We apply Lemma~\ref{lem-variational-localized-eigenvalue} with
\[
V=\mathbb{Z}^D,\qquad \Omega=\Omega_R,\qquad A=A_R,
\]
endowed with counting measure and unit edge weights. Set
\[
\lambda_R=\lambda_{A_R}.
\]
Then there exists a nontrivial function
\(u_R\in X_{\mathbb{Z}^D}(\Omega_R)\) satisfying
\[
-\Delta_p^\omega u_R(x)
=
\lambda_R\mathbf{1}_{A_R}(x)|u_R(x)|^{p-2}u_R(x),
\qquad x\in\Omega_R.
\]
Thus, \(u_R\) is a nontrivial solution of
\[
\left\{
\begin{array}{ll}
-\Delta_p^\omega u_R(x)
=
w_R(x)|u_R(x)|^{p-2}u_R(x),
& x\in\Omega_R,\\[4pt]
u_R(x)=0,
& x\in\mathbb{Z}^D\setminus\Omega_R,
\end{array}
\right.
\]
with
\[
w_R(x)=\lambda_R\mathbf{1}_{A_R}(x),
\qquad x\in\Omega_R.
\]
Since \(E_p\ge 0\), we have \(\lambda_R\ge 0\), and hence \(w_R\ge 0\).

It remains to estimate \(\|w_R\|_{\ell^1(\Omega_R,\mu)}\). Since
\(\mu\equiv 1\) and \(w_R=\lambda_R\mathbf{1}_{A_R}\), we have
\begin{equation}\label{sharp-wR-norm}
\|w_R\|_{\ell^1(\Omega_R,\mu)}
=
\lambda_R\#A_R.
\end{equation}
To estimate \(\lambda_R\), define
\[
\eta_R(x)
=
\left(1-\frac{|x|_1}{\rho_R}\right)_+,
\qquad x\in\mathbb{Z}^D.
\]
Since \(\rho_R<R\), we have \(\eta_R\in X_{\mathbb{Z}^D}(\Omega_R)\), and
\(\operatorname{supp}\eta_R\subset A_R\). Moreover, \(\eta_R\not\equiv 0\), so
\(\eta_R\) is admissible in the definition of \(\lambda_R\). Hence
\begin{equation}\label{sharp-lambda-test}
\lambda_R
\le
\frac{E_p(\eta_R)}
{\displaystyle\sum_{x\in A_R}|\eta_R(x)|^p}.
\end{equation}
For \(x\sim y\), we have
\[
\bigl||x|_1-|y|_1\bigr|\le 1,
\]
and therefore
\[
|\eta_R(y)-\eta_R(x)|
\le
\frac{1}{\rho_R}.
\]
Since \(\eta_R\) is supported in \(A_R\), only edges touching \(A_R\) contribute
to \(E_p(\eta_R)\). By the bounded degree of \(\mathbb{Z}^D\) and
\eqref{sharp-volume-AR},
\[
\#\bigl\{\{x,y\}:x\sim y,\ \{x,y\}\cap A_R\neq\emptyset\bigr\}
\le C\rho_R^D.
\]
Consequently,
\begin{equation}\label{sharp-test-energy-estimate}
E_p(\eta_R)
\le
C\rho_R^D\rho_R^{-p}
=
C\rho_R^{D-p}.
\end{equation}
On the other hand, \(\eta_R\ge 1/2\) in \(B(0,\rho_R/2)\). Hence, again by the
volume growth of balls in \(\mathbb{Z}^D\),
\begin{equation}\label{sharp-test-denominator-estimate}
\sum_{x\in A_R}|\eta_R(x)|^p
\ge
2^{-p}\#B(0,\rho_R/2)
\ge
C\rho_R^D.
\end{equation}
Combining \eqref{sharp-lambda-test}, \eqref{sharp-test-energy-estimate}, and
\eqref{sharp-test-denominator-estimate}, we obtain
\[
\lambda_R\le C\rho_R^{-p}.
\]
Therefore, by \eqref{sharp-volume-AR} and \eqref{sharp-wR-norm},
\[
\|w_R\|_{\ell^1(\Omega_R,\mu)}
=
\lambda_R\#A_R
\le
C\rho_R^{-p}\rho_R^D
=
C\rho_R^{-(p-D)}
=
C R^{-\alpha(p-D)}.
\]
Since \(\alpha(p-D)>\beta\), we have
\[
R^{-\alpha(p-D)}
=
o\bigl((R+1)^{-\beta}\bigr)
\qquad \text{as } R\to\infty.
\]
Thus, there exists \(R_0=R_0(\beta,C_*,p,D)>1\) such that, for every
\(R\in\mathbb{N}\) with \(R\ge R_0\),
\[
\|w_R\|_{\ell^1(\Omega_R,\mu)}
<
\frac{C_*}{(R+1)^\beta}.
\]
Since \(r_{\Omega_R}=R+1\), we obtain
\[
\|w_R\|_{\ell^1(\Omega_R,\mu)}
<
\frac{C_*}{r_{\Omega_R}^{\beta}}.
\]
This proves \eqref{inv-ineq}.
\end{proof}

\subsection{Capacitary estimates on radial path graphs}

In this subsection, we prove Propositions~\ref{prop-radial-path-cap-scales}
and~\ref{prop-radial-path-sharp}.

\begin{proof}[Proof of Proposition~\ref{prop-radial-path-cap-scales}]
Set
\[
T=\tau_\Omega\ge 2.
\]
First, we prove a basic capacitary estimate that will be used in the cases
\(p>D\) and \(p=D\).

Fix \(x\in\Omega\), and let \(\varphi\in X_V(\Omega)\) be such that
\(\varphi(x)=1\). By the definition of \(T\), we have
\(T+1\in V\setminus\Omega\). Hence
\[
\varphi(T+1)=0.
\]
Therefore,
\[
1
=
|\varphi(x)-\varphi(T+1)|
\le
\sum_{n=x}^{T}|\varphi(n+1)-\varphi(n)|.
\]
Using H\"older's inequality with
\[
c_n=(n+1)^{D-1},\qquad n\ge 0,
\]
we obtain
\[
1
\le
\left(
\sum_{n=x}^{T} c_n |\varphi(n+1)-\varphi(n)|^p
\right)^{1/p}
\left(
\sum_{n=x}^{T} c_n^{-\frac{1}{p-1}}
\right)^{\frac{p-1}{p}}.
\]
Since
\[
E_p(\varphi)
\ge
\sum_{n=x}^{T} c_n |\varphi(n+1)-\varphi(n)|^p,
\]
it follows that
\[
1
\le
E_p(\varphi)^{1/p}
\left(
\sum_{n=x}^{T} (n+1)^{-\frac{D-1}{p-1}}
\right)^{\frac{p-1}{p}}.
\]
Hence
\begin{equation}\label{general-est}
E_p(\varphi)
\ge
\left(
\sum_{n=x}^{T} (n+1)^{-\frac{D-1}{p-1}}
\right)^{1-p}.
\end{equation}

First, we prove {\rm(i)}. Assume that \(p>D\). Then
\[
0<\frac{D-1}{p-1}<1.
\]
Therefore, by the integral comparison and since \(T\ge 2\),
\[
\sum_{n=x}^{T} (n+1)^{-\frac{D-1}{p-1}}
\le
\sum_{n=0}^{T} (n+1)^{-\frac{D-1}{p-1}}
\le
C T^{1-\frac{D-1}{p-1}},
\]
where \(C=C(p,D)>0\). Since
\[
\left(1-\frac{D-1}{p-1}\right)(1-p)=-(p-D),
\]
it follows from \eqref{general-est} that
\[
E_p(\varphi)
\ge
C T^{-(p-D)}.
\]
Taking the infimum over all \(\varphi\in X_V(\Omega)\) such that
\(\varphi(x)=1\), we obtain
\[
\operatorname{Cap}_p(x,\Omega)
\ge
C T^{-(p-D)}.
\]
Since this estimate holds for every \(x\in\Omega\), it follows that
\[
\min_{x\in\Omega}\operatorname{Cap}_p(x,\Omega)
\ge
C T^{-(p-D)}.
\]
Using the definition of \(\mathcal{R}_p(\Omega)\), we conclude that
\[
\mathcal{R}_p(\Omega)
\le
C T^{p-D}
=
C\tau_\Omega^{p-D},
\]
for some \(C=C(p,D)>0\). This proves \eqref{est-cap1}.

Next, we prove {\rm(ii)}. Assume that \(p=D\). Then \eqref{general-est}
becomes
\[
E_p(\varphi)
\ge
\left(
\sum_{n=x}^{T}\frac{1}{n+1}
\right)^{1-p}.
\]
By the integral comparison and since \(T\ge 2\),
\[
\sum_{n=x}^{T}\frac{1}{n+1}
\le
\sum_{n=0}^{T}\frac{1}{n+1}
\le
C\log T,
\]
where \(C=C(D)>0\). Since \(1-p<0\), it follows that
\[
E_p(\varphi)
\ge
C(\log T)^{1-p}.
\]
Taking the infimum over all \(\varphi\in X_V(\Omega)\) such that
\(\varphi(x)=1\), we obtain
\[
\operatorname{Cap}_p(x,\Omega)
\ge
C(\log T)^{1-p}.
\]
Proceeding as above, this yields
\[
\mathcal{R}_p(\Omega)
\le
C(\log T)^{p-1}
=
C(\log \tau_\Omega)^{p-1},
\]
for some \(C=C(D)>0\). This proves \eqref{est-cap2}.

Finally, we prove {\rm(iii)}. Assume that \(1<p<D\) and \(s>D/p\). Let
\(u\in X_V(\Omega)\). For \(n\in\Omega\), since \(n\le T\) and \(u(T+1)=0\),
we have
\[
|u(n)|
=
|u(n)-u(T+1)|
\le
\sum_{k=n}^{T}|u(k+1)-u(k)|.
\]
Using H\"older's inequality with the weights \(c_k=(k+1)^{D-1}\), we obtain
\[
|u(n)|
\le
\left(
\sum_{k=n}^{T}c_k|u(k+1)-u(k)|^p
\right)^{1/p}
\left(
\sum_{k=n}^{T}c_k^{-\frac{1}{p-1}}
\right)^{\frac{p-1}{p}}.
\]
On the other hand,
\[
\sum_{k=n}^{T}c_k|u(k+1)-u(k)|^p
\le
E_p(u).
\]
Consequently,
\[
|u(n)|^p
\le
E_p(u)
\left(
\sum_{k=n}^{T}(k+1)^{-\frac{D-1}{p-1}}
\right)^{p-1}.
\]
Since \(1<p<D\), by the integral comparison,
\[
\sum_{k=n}^{T}(k+1)^{-\frac{D-1}{p-1}}
\le
\sum_{k=n}^{\infty}(k+1)^{-\frac{D-1}{p-1}}
\le
C(n+1)^{1-\frac{D-1}{p-1}},
\]
where \(C=C(p,D)>0\). Hence
\[
|u(n)|^p
\le
C E_p(u)(n+1)^{p-D},
\qquad n\in\Omega.
\]
Using \(\mu(n)=(n+1)^{D-1}\), we obtain
\[
\begin{aligned}
\|u\|_{\ell^{ps'}(\Omega,\mu)}^{ps'}
&=
\sum_{n\in\Omega}\mu(n)|u(n)|^{ps'}  \\
&\le
C E_p(u)^{s'}
\sum_{n=0}^{T}(n+1)^{D-1+(p-D)s'},
\end{aligned}
\]
where \(C=C(p,D,s)>0\). Since \(s>D/p\), we have
\[
D+(p-D)s'
=
s'\left(p-\frac{D}{s}\right)>0.
\]
Thus, by the integral comparison and since \(T\ge 2\),
\[
\sum_{n=0}^{T}(n+1)^{D-1+(p-D)s'}
\le
C T^{D+(p-D)s'}.
\]
Consequently,
\[
\|u\|_{\ell^{ps'}(\Omega,\mu)}^{ps'}
\le
C E_p(u)^{s'}T^{D+(p-D)s'}.
\]
Taking the power \(1/s'\), we obtain
\[
\|u\|_{\ell^{ps'}(\Omega,\mu)}^{p}
\le
C E_p(u)T^{p-\frac{D}{s}}.
\]
Therefore,
\[
\frac{\|u\|_{\ell^{ps'}(\Omega,\mu)}^{p}}{E_p(u)}
\le
C T^{p-\frac{D}{s}},
\qquad u\in X_V(\Omega),\ u\not\equiv 0.
\]
Taking the supremum over all nonzero \(u\in X_V(\Omega)\), we obtain
\[
\mathcal{R}_{p,s}(\Omega)
\le
C T^{p-\frac{D}{s}}
=
C\tau_\Omega^{p-\frac{D}{s}},
\]
for some \(C=C(p,D,s)>0\). This proves \eqref{est-cap3} and completes the
proof.
\end{proof}

\begin{proof}[Proof of Proposition~\ref{prop-radial-path-sharp}]
Throughout the proof, \(C>0\) denotes a generic constant independent of \(R\).

For \(R\in\mathbb{N}\) with \(R\ge 2\), set
\[
\Omega_R=B(0,R)=\{0,\ldots,R\}.
\]
For any \(1<p<\infty\) and any nonempty subset \(A\subset\Omega_R\), we apply
Lemma~\ref{lem-variational-localized-eigenvalue} and write
\[
\lambda_R(A)=\lambda_A.
\]
Then there exists a nontrivial function \(u_{R,A}\in X_V(\Omega_R)\) satisfying
\[
-\Delta_p^\omega u_{R,A}(n)
=
\lambda_R(A)\mathbf{1}_{A}(n)|u_{R,A}(n)|^{p-2}u_{R,A}(n),
\qquad n\in\Omega_R.
\]
Thus, with
\begin{equation}\label{w-R-A}
w_{R,A}(n)=\lambda_R(A)\mathbf{1}_{A}(n),
\qquad n\in\Omega_R,
\end{equation}
we obtain a nonnegative potential \(w_{R,A}\) and a nontrivial solution
\(u_{R,A}\) of \eqref{PP} in \(\Omega_R\).

For notational simplicity, once \(A\) is chosen in each case below, we write
\[
u_R=u_{R,A},
\qquad
w_R=w_{R,A}.
\]

We first record an estimate for \(\lambda_R(\Omega_R)\), valid for every
\(1<p<\infty\). Define
\[
\eta_R(n)=\left(1-\frac{n}{R+1}\right)_+,
\qquad n\in V.
\]
Then \(\eta_R(0)=1\), and
\[
\eta_R(n)=0,
\qquad n\ge R+1.
\]
Hence \(\operatorname{supp}\eta_R\subset\Omega_R\), and so
\(\eta_R\in X_V(\Omega_R)\). In particular, \(\eta_R\not\equiv 0\). Moreover,
\[
|\eta_R(n+1)-\eta_R(n)|=\frac{1}{R+1},
\qquad n=0,\ldots,R.
\]
Hence, by the definition of \(E_p\), the symmetry of the weights, and the
identity
\[
\omega_{n,n+1}=(n+1)^{D-1},
\]
we obtain
\begin{align}\label{Ep-est-casei}
\nonumber
E_p(\eta_R)
&=
\frac12
\sum_{\substack{m,n\in V\\ m\sim n}}
\omega_{mn}|\eta_R(n)-\eta_R(m)|^p  \\
\nonumber
&=
\sum_{n=0}^{R}\omega_{n,n+1}
|\eta_R(n+1)-\eta_R(n)|^p  \\
\nonumber
&=
(R+1)^{-p}\sum_{n=0}^{R}(n+1)^{D-1}  \\
&\le
C R^{D-p}.
\end{align}
On the other hand, \(\eta_R\ge 1/2\) for \(0\le n\le R/2\). Hence
\begin{equation}\label{sum-est-casei}
\sum_{n\in\Omega_R}\mu(n)|\eta_R(n)|^p
\ge
C\sum_{n=0}^{\lfloor R/2\rfloor}(n+1)^{D-1}
\ge
C R^D.
\end{equation}
Since \(\eta_R\) is admissible in the definition of \(\lambda_R(\Omega_R)\), it
follows from \eqref{Ep-est-casei} and \eqref{sum-est-casei} that
\begin{equation}\label{radial-lambda-full}
\lambda_R(\Omega_R)\le C R^{-p}.
\end{equation}

First, we prove {\rm(i)}. Let \(p>D\), \(\beta<p-D\), and \(C_*>0\). We choose
\(A=\Omega_R\). Then, by \eqref{w-R-A},
\[
w_R=\lambda_R(\Omega_R)\mathbf{1}_{\Omega_R}.
\]
Using \eqref{radial-lambda-full}, we obtain
\[
\|w_R\|_{\ell^1(\Omega_R,\mu)}
=
\lambda_R(\Omega_R)\mu(\Omega_R)
\le
C R^{-p}R^D
=
C R^{-(p-D)}.
\]
Since \(\beta<p-D\), we have
\[
R^{-(p-D)}=o(R^{-\beta})
\qquad \text{as } R\to\infty.
\]
Thus, there exists \(R_0=R_0(\beta,C_*,p,D)>2\) such that, for every
\(R\in\mathbb{N}\) with \(R\ge R_0\),
\[
\|w_R\|_{\ell^1(\Omega_R,\mu)}
<
C_*R^{-\beta}.
\]
This proves \eqref{sharp-radial-1}.

Next, we prove {\rm(ii)}. Let \(p=D\), \(\beta<p-1\), and \(C_*>0\). We
choose \(A=\{0\}\) and define
\[
\zeta_R(n)
=
\left(1-\frac{\log(n+1)}{\log(R+2)}\right)_+,
\qquad n\in V.
\]
Then \(\zeta_R(0)=1\). Moreover,
\[
\zeta_R(n)=0,
\qquad n\ge R+1.
\]
Hence \(\operatorname{supp}\zeta_R\subset\Omega_R\), and so
\(\zeta_R\in X_V(\Omega_R)\). In particular, \(\zeta_R\not\equiv 0\).
Moreover, for \(n=0,\ldots,R\),
\[
|\zeta_R(n+1)-\zeta_R(n)|
\le
\frac{C}{(n+1)\log R}.
\]
Since \(p=D\), it follows that
\[
\begin{aligned}
E_p(\zeta_R)
&\le
C(\log R)^{-p}
\sum_{n=0}^{R}(n+1)^{D-1-p}  \\
&=
C(\log R)^{-p}
\sum_{n=0}^{R}\frac{1}{n+1} \\
&\le
C(\log R)^{1-p}.
\end{aligned}
\]
Since
\[
\mu(0)\zeta_R(0)^p=1,
\]
the function \(\zeta_R\) is admissible in the definition of
\(\lambda_R(\{0\})\). Hence
\[
\lambda_R(\{0\})\le C(\log R)^{1-p}.
\]
Then, with \(w_R=\lambda_R(\{0\})\mathbf{1}_{\{0\}}\), we have
\[
\|w_R\|_{\ell^1(\Omega_R,\mu)}
=
\lambda_R(\{0\})\mu(0)
\le
C(\log R)^{1-p}.
\]
Since \(\beta<p-1\), we have
\[
(\log R)^{1-p}=o\bigl((\log R)^{-\beta}\bigr)
\qquad \text{as } R\to\infty.
\]
Thus, there exists \(R_0=R_0(\beta,C_*,D)>2\) such that, for every
\(R\in\mathbb{N}\) with \(R\ge R_0\),
\[
\|w_R\|_{\ell^1(\Omega_R,\mu)}
<
C_*(\log R)^{-\beta}.
\]
This proves \eqref{sharp-radial-2}.

Finally, we prove {\rm(iii)}. Assume that \(1<p<D\) and \(s>D/p\). Let
\(\beta<p-\frac{D}{s}\) and \(C_*>0\). Again, we choose \(A=\Omega_R\). Then
\[
w_R=\lambda_R(\Omega_R)\mathbf{1}_{\Omega_R}.
\]
Using \eqref{radial-lambda-full}, we obtain
\[
\begin{aligned}
\|w_R\|_{\ell^s(\Omega_R,\mu)}
&=
\lambda_R(\Omega_R)\mu(\Omega_R)^{1/s}  \\
&\le
C R^{-p}R^{D/s} \\
&=
C R^{-\left(p-\frac{D}{s}\right)}.
\end{aligned}
\]
Since \(\beta<p-\frac{D}{s}\), we have
\[
R^{-\left(p-\frac{D}{s}\right)}
=
o(R^{-\beta})
\qquad \text{as } R\to\infty.
\]
Thus, there exists \(R_0=R_0(\beta,C_*,p,D,s)>2\) such that, for every
\(R\in\mathbb{N}\) with \(R\ge R_0\),
\[
\|w_R\|_{\ell^s(\Omega_R,\mu)}
<
C_*R^{-\beta}.
\]
This proves \eqref{sharp-radial-3} and completes the proof.
\end{proof}

\subsection{A capacitary estimate on regular trees}

In this subsection, we prove Propositions~\ref{prop-tree-cap} and \ref{prop-tree-sharp}. 

\begin{proof}[Proof of Proposition~\ref{prop-tree-cap}]
Fix \(o\in\Omega\). For \(n\in\mathbb{N}_0\), set
\[
S_n=\{x\in V:\ d(o,x)=n\},
\]
and let \(\mathcal{E}_n\) be the set of oriented edges joining \(S_n\) to
\(S_{n+1}\), namely
\[
\mathcal{E}_n
=
\bigl\{(z,z')\in S_n\times S_{n+1}:\ z\sim z'\bigr\}.
\]
Since \(V\) is the \(q\)-regular tree, the vertex \(o\) has \(q\) neighbors,
and every vertex in \(S_n\), \(n\ge 1\), has exactly one neighbor in
\(S_{n-1}\) and \(q-1\) neighbors in \(S_{n+1}\). Hence
\begin{equation}\label{ppt1}
\#S_0=1,\qquad
\#S_n=q(q-1)^{n-1},\qquad n\ge 1,
\end{equation}
and
\begin{equation}\label{ppt2}
\#\mathcal{E}_n=q(q-1)^n,\qquad n\in\mathbb{N}_0.
\end{equation}

Let \(\varphi\in X_V(\Omega)\) be such that \(\varphi(o)=1\). Since
\(\varphi\) has finite support, there exists \(N\ge 1\) such that
\[
\operatorname{supp}\varphi\subset B(o,N-1).
\]
In particular, \(\varphi=0\) in \(S_N\). For each \(y\in S_N\), since \(V\) is
a tree, there exists a unique path joining \(o\) to \(y\). Since \(d(o,y)=N\),
we write this path as
\[
o=x_0(y)\sim x_1(y)\sim\cdots\sim x_N(y)=y.
\]
Then
\[
x_i(y)\in S_i,\qquad i=0,\ldots,N.
\]
Since \(\varphi(o)=1\) and \(\varphi(y)=0\), we have
\[
1
=
|\varphi(o)-\varphi(y)|
\le
\sum_{n=0}^{N-1}
|\varphi(x_{n+1}(y))-\varphi(x_n(y))|.
\]
Summing this inequality over \(y\in S_N\) and dividing by \(\#S_N\), we obtain
\begin{equation}\label{sum1}
1
\le
\frac{1}{\#S_N}
\sum_{y\in S_N}
\sum_{n=0}^{N-1}
|\varphi(x_{n+1}(y))-\varphi(x_n(y))|.
\end{equation}
Since \(x_n(y)\in S_n\) and \(x_{n+1}(y)\in S_{n+1}\), the oriented edge
\((x_n(y),x_{n+1}(y))\) belongs to \(\mathcal{E}_n\). Moreover, for each
\(0\le n\le N-1\), every edge in \(\mathcal{E}_n\) belongs to exactly
\((q-1)^{N-n-1}\) of the unique paths joining \(o\) to the vertices of
\(S_N\). Therefore, by \eqref{ppt1} and \eqref{ppt2},
\[
\begin{aligned}
\frac{1}{\#S_N}
\sum_{y\in S_N}
\sum_{n=0}^{N-1}
|\varphi(x_{n+1}(y))-\varphi(x_n(y))|
&=
\sum_{n=0}^{N-1}
\frac{(q-1)^{N-n-1}}{\#S_N}
\sum_{(z,z')\in\mathcal{E}_n}
|\varphi(z')-\varphi(z)|  \\
&=
\sum_{n=0}^{N-1}
\frac{1}{\#\mathcal{E}_n}
\sum_{(z,z')\in\mathcal{E}_n}
|\varphi(z')-\varphi(z)|.
\end{aligned}
\]
Hence, by \eqref{sum1},
\[
1
\le
\sum_{n=0}^{N-1}
\frac{1}{\#\mathcal{E}_n}
\sum_{(z,z')\in\mathcal{E}_n}
|\varphi(z')-\varphi(z)|.
\]
By H\"older's inequality,
\[
\begin{aligned}
1
&\le
\left(
\sum_{n=0}^{N-1}
\sum_{(z,z')\in\mathcal{E}_n}
|\varphi(z')-\varphi(z)|^p
\right)^{1/p}
\left(
\sum_{n=0}^{N-1}(\#\mathcal{E}_n)^{-\frac{1}{p-1}}
\right)^{\frac{p-1}{p}}  \\
&\le
E_p(\varphi)^{1/p}
\left(
\sum_{n=0}^{N-1}(\#\mathcal{E}_n)^{-\frac{1}{p-1}}
\right)^{\frac{p-1}{p}}.
\end{aligned}
\]
Therefore,
\[
E_p(\varphi)
\ge
\left(
\sum_{n=0}^{N-1}(\#\mathcal{E}_n)^{-\frac{1}{p-1}}
\right)^{1-p}.
\]
On the other hand, by \eqref{ppt2},
\[
\begin{aligned}
\sum_{n=0}^{N-1}(\#\mathcal{E}_n)^{-\frac{1}{p-1}}
&=
q^{-\frac{1}{p-1}}
\sum_{n=0}^{N-1}(q-1)^{-\frac{n}{p-1}}  \\
&\le
q^{-\frac{1}{p-1}}
\sum_{n=0}^{\infty}(q-1)^{-\frac{n}{p-1}}  \\
&=
\frac{q^{-\frac{1}{p-1}}}
{1-(q-1)^{-\frac{1}{p-1}}}.
\end{aligned}
\]
Hence
\[
E_p(\varphi)
\ge
q\left(1-(q-1)^{-\frac{1}{p-1}}\right)^{p-1}
=
c_{p,q}.
\]
Taking the infimum over all \(\varphi\in X_V(\Omega)\) such that
\(\varphi(o)=1\), we obtain
\[
\operatorname{Cap}_p(o,\Omega)\ge c_{p,q}.
\]
Since \(o\in\Omega\) was arbitrary,
\[
\min_{o\in\Omega}\operatorname{Cap}_p(o,\Omega)\ge c_{p,q}.
\]
Using the definition of \(\mathcal{R}_p(\Omega)\), we conclude that
\[
\mathcal{R}_p(\Omega)
=
\max_{o\in\Omega}\operatorname{Cap}_p(o,\Omega)^{-1}
\le
\frac{1}{c_{p,q}}.
\]
This proves \eqref{cap-RT}.
\end{proof}

\begin{proof}[Proof of Proposition~\ref{prop-tree-sharp}]
Throughout the proof, \(C>0\) denotes a generic constant independent of \(R\). Let \(C_*>c_{p,q}\) be fixed.

Let \(R\in\mathbb{N}\), and set
\[
\Omega_R=B(o,R).
\]
First, we compute the point \(p\)-capacity of \(o\) relative to \(\Omega_R\).  For
\(n\in\mathbb{N}_0\), set
\[
S_n=\{x\in V:\ d(o,x)=n\},
\]
and let \(\mathcal{E}_n\) be the set of oriented edges joining \(S_n\) to
\(S_{n+1}\).  Set
\[
A_R
=
\sum_{n=0}^{R}
(\#\mathcal{E}_n)^{-\frac{1}{p-1}}.
\]

We claim that
\begin{equation}\label{tree-cap-exact}
\operatorname{Cap}_p(o,\Omega_R)=A_R^{1-p}.
\end{equation}
Indeed, let \(\varphi\in X_V(\Omega_R)\) be such that \(\varphi(o)=1\). Since
\(\varphi=0\) in \(V\setminus\Omega_R\), we have \(\varphi=0\) in
\(S_{R+1}\). Repeating the argument used in the proof of
Proposition~\ref{prop-tree-cap}, with \(N=R+1\), gives
\[
E_p(\varphi)
\ge
\left(
\sum_{n=0}^{R}
(\#\mathcal{E}_n)^{-\frac{1}{p-1}}
\right)^{1-p}
=
A_R^{1-p}.
\]
Taking the infimum over all such \(\varphi\), we obtain
\begin{equation}\label{cap-pr1}
\operatorname{Cap}_p(o,\Omega_R)\ge A_R^{1-p}.
\end{equation}
Conversely, define
\[
\psi_R(x)
=
\frac{\displaystyle\sum_{k=d(o,x)}^{R}
(\#\mathcal{E}_k)^{-\frac{1}{p-1}}}
{\displaystyle\sum_{k=0}^{R}
(\#\mathcal{E}_k)^{-\frac{1}{p-1}}},
\qquad x\in\Omega_R,
\]
and set \(\psi_R=0\) in \(V\setminus\Omega_R\).  Then
\(\psi_R\in X_V(\Omega_R)\) and \(\psi_R(o)=1\).  Moreover, if
\((z,z')\in\mathcal{E}_n\), \(0\le n\le R\), then
\[
|\psi_R(z')-\psi_R(z)|
=
\frac{(\#\mathcal{E}_n)^{-\frac{1}{p-1}}}{A_R}.
\]
Therefore,
\[
\begin{aligned}
E_p(\psi_R)
&=
\sum_{n=0}^{R}
\sum_{(z,z')\in\mathcal{E}_n}
|\psi_R(z')-\psi_R(z)|^p  \\
&=
A_R^{-p}
\sum_{n=0}^{R}
(\#\mathcal{E}_n)^{-\frac{1}{p-1}}
=
A_R^{1-p}.
\end{aligned}
\]
Hence
\begin{equation}\label{cap-pr2}
\operatorname{Cap}_p(o,\Omega_R)\le A_R^{1-p}.
\end{equation}
Combining \eqref{cap-pr1} and \eqref{cap-pr2}, we obtain \eqref{tree-cap-exact}.

Next, we apply Lemma~\ref{lem-variational-localized-eigenvalue} with
\(\Omega=\Omega_R\) and \(A=\{o\}\). It gives a nontrivial function \(u_R\in X_V(\Omega_R)\) satisfying
\[
-\Delta_p^\omega u_R(x)
=
\lambda_{\{o\}}\mathbf{1}_{\{o\}}(x)|u_R(x)|^{p-2}u_R(x),
\qquad x\in\Omega_R,
\]
where
\[
\lambda_{\{o\}}
=
\inf_{\substack{u\in X_V(\Omega_R)\\ u(o)\neq 0}}
\frac{E_p(u)}{\mu(o)|u(o)|^p}.
\]
Thus, setting
\[
w_R(x)=\lambda_{\{o\}}\mathbf{1}_{\{o\}}(x),
\qquad x\in\Omega_R,
\]
we obtain a nonnegative potential \(w_R\) and a nontrivial solution \(u_R\) of
\eqref{PP} in \(\Omega_R\).

On the other hand, by the definition of the point capacity,
\[
\lambda_{\{o\}}
=
\frac{\operatorname{Cap}_p(o,\Omega_R)}{\mu(o)}.
\]
Therefore,
\[
\|w_R\|_{\ell^1(\Omega_R,\mu)}
=
\lambda_{\{o\}}\mu(o)
=
\operatorname{Cap}_p(o,\Omega_R).
\]
It remains to estimate this capacity. By \eqref{tree-cap-exact} and \eqref{ppt2}, 
\[
\operatorname{Cap}_p(o,\Omega_R)
=
\left(
q^{-\frac{1}{p-1}}
\sum_{n=0}^{R}
(q-1)^{-\frac{n}{p-1}}
\right)^{1-p}.
\]
As \(R\to\infty\),
\[
q^{-\frac{1}{p-1}}
\sum_{n=0}^{R}
(q-1)^{-\frac{n}{p-1}}
\longrightarrow
\frac{q^{-\frac{1}{p-1}}}
{1-(q-1)^{-\frac{1}{p-1}}}.
\]
Therefore,
\[
\operatorname{Cap}_p(o,\Omega_R)
\longrightarrow
q\left(1-(q-1)^{-\frac{1}{p-1}}\right)^{p-1}
=
c_{p,q}.
\]
Since \(C_*>c_{p,q}\), there exists \(R_0=R_0(C_*,p,q)>1\) such that, for
every \(R\in\mathbb{N}\) with \(R\ge R_0\),
\[
\|w_R\|_{\ell^1(\Omega_R),\mu)}
=
\operatorname{Cap}_p(o,\Omega_R)
<
C_*.
\]
This completes the proof.
\end{proof}

\end{document}